\documentclass[12pt]{iopart}
\usepackage{iopams}  
  \expandafter\let\csname equation*\endcsname\relax
  \expandafter\let\csname endequation*\endcsname\relax
\usepackage{graphicx}
\usepackage{amssymb,amsthm,amsmath}
\usepackage{url}

\newtheorem{theorem}{Theorem}[section]
\newtheorem{definition}{Definition}[section]
\newtheorem{lemma}{Lemma}

\usepackage{epsfig,subfigure} 
\usepackage{fullpage}
 \usepackage{algorithm}
 \usepackage{algpseudocode}

\usepackage{epsfig,subfigure} 
\usepackage{fullpage}
 \usepackage{algorithm}
 \usepackage{algpseudocode}

\newcommand{\dis}[1]{\displaystyle{#1}}

\newcommand{\Z}{\mathbb{Z}}

\newcommand{\Cb}{C}
\newcommand{\bfone}{{\bf1}}
\newcommand{\bfmu}{{\boldsymbol \mu}}

\newcommand{\st}{\text{ such that }}

\newcommand{\lra}[1]{\langle #1 \rangle}
\newcommand{\abs}[1]{\left\lvert #1 \right\rvert}
\newcommand{\norm}[1]{\left\lvert \abs{#1} \right\rvert}
\newcommand{\paren}[1]{\left( #1 \right)}

\newcommand{\bo}[1]{\boldsymbol{#1}}
\newcommand{\eps}{\varepsilon}

\DeclareMathOperator*{\argmin}{argmin}

\newcommand{\subs}[1]{{#1_i^{(n)}}}
\newcommand{\subj}[1]{{#1_j^{(n)}}}
\newcommand{\sub}[1]{{#1^{(n)}}}
\newcommand{\subl}[1]{{#1^{(N)}}}

\newcommand{\Nn}{\mathcal{N}}

\newtheorem{prop}{Proposition}

\theoremstyle{definition}
\theoremstyle{remark}
\newtheorem*{remark*}{Remark}
\renewcommand{\refname}

\begin{document}
\title{Efficient Estimation of Regularization Parameters via Downsampling and the Singular Value Expansion}
\author{Rosemary  A. Renaut$^1$,   
Michael Horst$^2$,   
Yang Wang$^3$,   
 Douglas Cochran$^4$ and  
 Jakob Hansen$^5$}
  \date{\today}
   \address{$^1$School of Mathematical and Statistical Sciences, Arizona State University, P.O. Box 871804, Tempe, AZ 85287-1804, \\
$^2$  Department of Mathematics,  The Ohio State University, 100 Math Tower, 231 West 18th Avenue, Columbus, OH 43210-1174, \\
$^3$ Department of Mathematics, The Hong Kong University of Science and Technology, Clear Water Bay, Kowloon, Hong Kong,\\
$^4$ School of Electrical, Computer and Energy Engineering, Ira A. Fulton Schools of Engineering, Arizona State University, P.O. Box 875706,  Tempe, AZ 85287-5706,\\
$^5$David Rittenhouse Laboratory, University of Pennsylvania, 209 S. 33rd St., Philadelphia, PA 19104,}
 \eads{\mailto{renaut@asu.edu}, \mailto{horst.59@osu.edu}, \mailto{yangwang@ust.hk}, \mailto{cochran@asu.edu}, \mailto{jhansen@math.upenn.edu}}

\begin{abstract}
The solution, $\bo x$, of the linear system of equations $A\bo x\approx \bo b$ arising from the discretization of an ill-posed integral equation with a square integrable kernel $H(s,t)$ is considered. The Tikhonov regularized solution $\bo x(\lambda)$  is found   as the minimizer of $J(\bo x)=\{ \|A \bo x -\bo b\|_2^2 + \lambda^2 \|L \bo x\|_2^2\}$. $\bo x(\lambda)$  depends on  regularization parameter  $\lambda$ that trades off the data fidelity,  and on the smoothing norm  determined by $L$.  Here we consider the case where $L$ is diagonal and invertible, and employ  the Galerkin method to provide the relationship between the singular value expansion and the singular value decomposition for square integrable kernels.  The  resulting approximation of the  integral equation  permits examination of  the properties of the regularized solution $\bo x(\lambda)$ independent of the sample size of the data. We prove that estimation of the  regularization parameter can   be obtained by consistently down sampling the data and the system matrix, leading to solutions of coarse to fine grained resolution.  Hence, the estimate of $\lambda$ for a large problem may be found by downsampling to a smaller problem, or to a set of smaller problems, effectively moving the costly estimate of the regularization parameter to the coarse representation of the problem.    Moreover, the full singular value decomposition for the fine scale system is replaced by a number of dominant terms  which is determined from the coarse resolution system, again reducing the computational cost. Numerical results illustrate the theory and  demonstrate the practicality of the approach for regularization parameter estimation using generalized cross validation, unbiased predictive risk estimation and the discrepancy principle applied for both the system of equations, and the augmented system of equations. 
\end{abstract}

\noindent\textbf{Keywords:} Singular value expansion, Singular value decomposition, Ill-posed inverse problem, Tikhonov regularization, Regularization parameter estimation  

\ams{65F22, 45B05}
\maketitle

\section{Introduction}
We consider numerical solutions of the Fredholm integral equation of the first kind 
\begin{align}
\int_{\Omega_t} H(s,t)f(t) \, dt = g(s), \quad s \in \Omega_s,  \label{eq:fred}
\end{align}
for real functions $f$ and $g$    defined on domains $\Omega_t$ and $\Omega_s$, respectively. The source function $f(t)$ is unknown, $H(s,t)$ is a known square-integrable kernel,  i.e. $\|H\|=(\int_{\Omega_s}\int_{\Omega_t} H^2(s,t) dt ds )^{1/2} < \infty$,  and square integrable $g(s)$ is approximated from sampled data.  For ease we assume   that $H$ is non-degenerate. The theoretical analysis for the existence of solutions of the integral equation is well studied e.g. \cite{baker,smithies} and uses the singular value expansion (SVE)  for the kernel $H$. For  the inner product    defined by 
\begin{align*}
\lra{f,\psi} =: \int_{\Omega_t}  {f(t)}\psi(t) \, dt, \quad \|f\|= \lra{f,f} ^{1/2},
\end{align*}
the SVE for $H$ is given by the mean convergent expansion
\begin{align} \label{eq:sve}
H(s,t) = \sum_{i=1}^\infty\mu_iu_i(s) {v}_i(t).
\end{align}
The sets of orthonormal functions $\{u_i(s)\}$ and $\{v_i(t)\}$ are complete and comprise the  left and right singular functions of the kernel,  respectively, 
\begin{align*}
\lra{H^*(s,t),u_i(s)} &= \int_{\Omega_s} H(s,t) u_i(s) \, ds  = \mu_iv_i(t),\\
\lra{H(s,t),v_i(t)} &= \int_{\Omega_t} H(s,t)v_i(t)\, dt  = \mu_iu_i(s).\end{align*}
The $\mu_i$, ordered such that $\mu_1\ge\mu_2\ge\ldots>0$, where positivity follows by the assumption of non degeneracy, are the singular values of $H$. By completeness of the singular functions and the square integrability of the kernel 
\begin{align*}
 \norm{H}_2^2 &= \sum\limits_{i=1}^\infty \mu_i^2<\infty.
 \end{align*}
 
Applying the SVE in \eqref{eq:fred} yields the mean convergent  expansion for $f$
 \begin{align}\label{contsoln}
   f(t) = \sum_{i=1}^\infty \frac{\lra{u_i(s),g}}{\mu_i} v_i(t) =:\sum_{i=1}^\infty \frac{\hat{g_i}}{\mu_i} v_i(t), 
 \end{align} 
which exists and is  square integrable  if and only if 
 \begin{align*}
 \sum_{i=1}^{\infty} \left(\frac{\hat{g_i}}{\mu_i}\right)^2 < \infty,
 \end{align*}
 cf. \cite{smithies}.
 It is immediate that  for square integrability of $f$  the coefficients ${\hat{g_i}}$ must eventually decay faster than the singular values. This requirement is known as  the {\em continuous Picard condition} \cite{Hansen}. If $g$ is error contaminated such that ultimately   $  {\hat{g_i}}$ roughly stagnate at the error level, the Picard condition will be violated. 
In particular, the  determination of $f$ from    $g$ is an ill-posed problem, even if a unique solution exists   $f$ will be  sensitive to errors in  $g$.  Additional constraint, or  stabilization, conditions are needed in order to estimate a suitable square integrable solution, e.g. by truncating the solution, and/or regularization by filtering, also known as Tikhonov regularization. Applying the Galerkin method for the numerical solution of the integral equation  provides a discretization for which the singular value decomposition (SVD) of the underlying system matrix is closely related to the SVE of the kernel  \cite{hansensvepaper}. The  properties of the discrete solution are closely related to those of the continuous solution, and hence  regularization of the discrete solution effectively regularizes the continuous solution for a sufficiently fine resolution. The regularization of the discrete system is the topic of this study. 

Vogel \cite[Chapter 7]{vogel:02}, analyzed  methods for selecting the  regularization parameter, denoted by $\lambda$,   under assumption of the existence of a singular system for the partially discrete operator mapping from the infinite dimensional to the discrete system of size $n$.  With the assumption of discrete data and stochastic noise,  the convergence for the error in the solution as an approximation to the true infinite dimensional solution, and the convergence of the predictive error for the solution, defined as how well the given solution predicts the true data, were examined. The results  employ   the singular system    to obtain the discrete solution in terms of the singular system, and hence also provide  estimates for the error, the predictive error, and the residual in terms of this system. The analysis explicitly includes the noise terms in the data. Assuming algebraic decay of the singular values and the coefficients of the true solution, at given algebraic rates, and that the orthogonal projection of the true solution onto the null space of the partially discrete operator vanishes as $n$  increases, yields asymptotic rates of convergence for the expected error norms, and predictive errors, for the truncated singular value decomposition (TSVD) solution to the discrete problem, i.e. regularizing by truncation, as well as for Tikhonov regularization. This formulation permits   comparison of the regularization methods, in terms of the convergence of the estimated $\lambda$  with $n$  and assumed rate parameters. In contrast, here we exploit the SVE-SVD relationship, with the assumption of   square integrability of $H(s,t)$, $f(t)$ and $g(s)$, to derive convergence results for  $\lambda$ with $n$, but without assuming specific rates of convergence of $\lambda$. Instead of the  TSVD regularization we consider the filtered TSVD solution, dependent on an effective numerical rank of the system.  As in Vogel \cite{vogel:02}, our results also consider the impact of noise in the data.

Although the connection between the regularized solutions of the continuous and discrete approximations is well-accepted, it appears that these convergence results  have not been extensively exploited practically in the context of efficient estimation of $\lambda$. Here we use the potentially nested set of  coarse to fine  discretizations of \eqref{eq:fred} by consistently defined   approximations to obtain $\lambda$ for a fine resolution solution, without applying estimation of $\lambda$  on the fine scale. This  approach   provides a cost effective mechanism for estimating $\lambda$ in the context of any  Tikhonov, or iterated Tikhonov \cite{SB,WoRo:07,Zhd}, regularization scheme for the solution of \eqref{eq:fred}, including in the presence of noise, assuming uniform sampling of the data.  The number of significant components in the SVD basis at the fine scale problem is determined from the coarse scale SVD basis, removing the high cost of forming the complete SVD for the high resolution system. Then, the  time consuming part of estimating $\lambda$ is played out  for the coarse scale representation of the problem only. 

  An outline of the paper now follows.  In order to use the approximate singular value expansion \cite{hansensvepaper} we  apply the Galerkin method for solving   \eqref{eq:fred}. This is briefly reviewed in section~\ref{sec:Galerkin}, and leads to the discretization formulae for the regularized discrete solution dependent on $\lambda$ and on sampling level $n$. Techniques for estimating $\lambda$ are reviewed in section~\ref{sec:regest}.   Many of these techniques require that some information about the noise contamination of the measurements is available, which leads to weighting, or left preconditioning, of the system matrix, as compared to right preconditioning  due to the invertible regularization operator $L$. We prove that the resulting left and right preconditioned   kernel is still square integrable;   section~\ref{mod kernel}. The convergence of $\lambda$ across scales is provided in section~\ref{sec:lambdaconv} based on theoretical results for the numerical rank in section~\ref{theory rank}. Practical implementation is described in section~\ref{sec:practical}, including the downsampling in section~\ref{sec:downsample}, estimating the numerical rank in section~\ref{sec:numericalrank} and the algorithm for parameter estimation in the presence of noise in section~\ref{sec:regparam}.    
  Numerical illustrations verifying the theoretical convergence results are provided in section~\ref{simulations}. In section~\ref{results} simulations for a smooth and a piece wise constant source $f(t)$, and an example with slowly decaying spectrum,  demonstrate the practicality of the technique.  Future work and conclusions are discussed in section~\ref{conclusions}.  
 
 \subsection{Notation}
  We first review the notation that is adopted throughout the paper. All variables in boldface $\bo {x}$, $\bo {b}$, etc. refer to vectors, with scalar entries e.g. $x_i$ distinguished from columns of a matrix $U$   given by $\bo u_i$. The proofs require the use of multiple resolutions for discretizing the functions. Any variable with a superscript $^{(n)}$ relates to that variable for a discretization with $n$ points, or an expansion with $n$ terms.  In general $\lambda$ is a regularization parameter, $A$ is a system matrix derived from  kernel $H$, the unknown source function is $f(t)$, samples of function $g(s)$ are assumed, and for any function $f$, $\hat{f}_i$ indicates the $i^{th}$ Galerkin coefficient of the function. For arbitrary vector $\bf y$ we denote $\bo y \sim \mathcal{N}(\bo y_0, C)$ to indicate that $\bo y$ is a random vector following a multivariate normal distribution with expected value, $E(\bo y) = \bo y_0$ and covariance matrix $C$. $L^2(\Omega)$ denotes the linear space  of square integrable functions on  $\Omega$. The weighted norm is given by $\| \bo x\|_W^2=\bo x^T W \bo x$.   For scalar $k$ the power  $\bo x^k$ indicates the component wise power for each component of vector $\bo x$. We also reserve $\zeta^2$ for the variance of white noise data. 

\section{Background Material}\label{sec:Galerkin}
\subsection{Approximating the Singular Value Expansion}
Suppose 
 $\{\phi_j(t)\}_{j=1}^\infty$ and $\{\psi_i(s)\}_{i=1}^\infty$ are orthonormal bases (ONB) for   $L^2\paren{\Omega_t}$  and  $L^2\paren{\Omega_s}$, respectively,  such that  
 \begin{align}\label{expansion}
  f(t) = \sum_{i=1}^\infty \lra{\phi_i(t), f(t)} \phi_i(t) \quad \mathrm{and} \quad g(s) = \sum_{i=1}^\infty \lra{\psi_i(s), g(s)} \psi_i(s).
 \end{align}
 The Galerkin method as a general discretization scheme for computing eigensystems was introduced in \cite[Section 3.8]{baker}, and extended as described in Algorithm~\ref{MM} for computing the SVE in \cite{hansensvepaper}.
 \begin{algorithm}
\caption{Galerkin Method for Approximating the SVE\label{MM} \cite{hansensvepaper}}
\begin{algorithmic}[1]
\Require    ONB    $\{\phi_j(t)\}_{j=1}^n$ and  $\{\psi_i(s)\}_{i=1}^n$, and kernel function $H(s,t)$.
\State  Calculate  the kernel matrix $A^{(n)}$ with entries $(a^{(n)}_{ij})$ 
\begin{align}\label{kernel matrix}
a^{(n)}_{ij} =: \lra{ \psi_i(s), \lra{H(s, t),  \phi_j(t)}} = \int_{\Omega_s}\int_{\Omega_t} \psi_i(s) H(s,t) \phi_j(t) dt  \,ds,\quad i,j =1:n.
\end{align}
\State Compute the SVD, $\sub{A} =\sub{U} \sub{\Sigma} (\sub{V})^T$, where $\sub{U}$ and $\sub{V}$ are orthogonal, 
\begin{align*}
\sub{U}= (u^{(n)}_{ij}), \,\,\sub{V}=(v^{(n)}_{ij}),\,\, \sub{\Sigma} =\mathrm{diag}(\sub{\sigma_1}, \dots, \sub{\sigma_n}), \,\, \sub{\sigma_1} \ge \dots \ge \sub{\sigma_n} >0.
\end{align*}
\State  Define 
\begin{align}\label{singvec}
\sub{\tilde{u}_i}(s) = \sum_{j=1}^n\sub{u_{ji}} \psi_j(s),  \quad \sub{\tilde{v}_i}(t) = \sum_{j=1}^n \sub{v_{ji}}\phi_j(t), \quad i=1:n.
\end{align}
\end{algorithmic}
\end{algorithm}

We give without proof the key results,  cf.  \cite[Theorems 1, 2, 4, 5]{hansensvepaper},  which relate the SVE and SVD singular systems. In particular, the singular values $\mu_i$, and singular functions $v_i(t)$, $u_i(s)$ for \eqref{eq:sve}  are approximated by $\sub{\sigma_i}$, $\sub{\tilde{v}_i}(t) $ and $\sub{\tilde{u}_i}(s) $, respectively, \cite{hansensvepaper}.
\begin{theorem}[SVE-SVD, \cite{hansensvepaper}]\label{SVE-SVD}
Suppose
\begin{align}\label{defDelta}
(\sub{\Delta})^2=:\|H\|^2 - \|\sub{A}\|_F^2.\end{align} Then the following hold for all $i$ and $n$, independent of the convergence of $\sub{\Delta}$ to $0$: 
\begin{enumerate}
\item{\label{thm1}} ${\sub{\sigma_i}} \le {\sigma^{(n+1)}_i} \le \mu_i$.
\item{\label{thm2}}   $0 \le \mu_i - {\sub{\sigma_i}} \le {\sub{\Delta}}$.
\item  ${\sub{\sigma_i}} \le \mu_i \le \sqrt{{(\sub{\sigma_i})^2}+{(\sub{\Delta})^2}}$.
\item{\label{thm4}}  If $\mu_i\ne\mu_{i+1}$, then $0 \le \max\{\|{u_i-{\sub{\tilde u_i}}}\|, \|{v_i-{\sub{\tilde v_i}}}\|\}  \le \sqrt{\frac{2 \, {\sub{\Delta}}}{\mu_i-\mu_{i+1}}}$.
\end{enumerate}
\end{theorem}
Taking the inner product in \eqref{eq:fred}    defines the discrete system of equations 
\begin{align}\label{discrete system}
\sub{\bo b} = \sub{A} \sub{\bo x}, \quad \text{with} \quad \sub{b_i}=\lra{\sub{\psi_i}(s), g(s)}, \quad \text{and}\quad \sub{x_i}=\lra{\sub{\phi_i}(t), f(t)},
\end{align}
with the truncated approximation to \eqref{expansion} 
\begin{align*}
f(t) \approx \sub{f}(t)=: \sum_{i=1}^n \sub{x_i} \phi_i(t) \quad \mathrm{and} \quad g(s) \approx \sub{g}(s)=: \sum_{i=1}^n \sub{b_i} \psi_i(s).
\end{align*} 
Solving \eqref{discrete system} for $\sub{\bo x}$ using the SVD for $\sub{A}$  yields the approximation \cite[(26)]{hansensvepaper} for \eqref{contsoln} 
\begin{align*}
 \sub{f}(t)&=  \sum_{j=1}^n\left(\sum_{i=1}^n \frac{(\sub{\bo u_i})^T \sub{\bo b}}{\sub{\sigma_i}} \sub{  v_{ji}}\right) \phi_j(t) \\
 &= \sum_{i=1}^n \frac{(\sub{\bo u_i})^T \sub{\bo b}}{\sub{\sigma_i}} \left(\sum_{j=1}^n\sub{ v_{ji}} \phi_j(t) \right) =  \sum_{i=1}^n \frac{(\sub{\bo u_i})^T \sub{\bo b}}{\sub{\sigma_i}}\sub{\tilde{v}_i}(t)\\
 &= \sum_{i=1}^n \frac{\sub{\beta_i}}{\sub{\sigma_i}} \sub{\tilde{v}_i}(t).
\end{align*} 
Here
\begin{align} \label{betacoeff}
\sub{\beta_i}=(\sub{\bo u_i})^T \sub{\bo b} = \lra{\sub{\tilde{u}_i}(s), g(s)} \approx \lra{u_i(s), g(s)} = \hat{g}_i,\end{align}  which follows from Theorem~\ref{SVE-SVD} and  the definitions in \eqref{singvec}.
Therefore, regularizing  through the introduction of filter factors $\sub{q_i}$, $i=1:n$, 
\begin{align}\label{solnfilter}
 \sub{f_{\text{Reg}}}(t)=: \sum_{i=1}^n \sub{q_i} \frac{\sub{\beta_i}}{\sub{\sigma_i}} \sub{\tilde{v}_i}(t), 
\end{align}
 yields an approximate regularized solution of the continuous function \cite{hansensvepaper}. Practically, we suppose $\sub{q_i}=\sub{q}(\sub{\lambda}, \sub{\sigma_i})$ for regularization parameter $\sub{\lambda}$ and desire to find a suitable choice for $\sub{\lambda}$ such that solution  $\sub{f_{\text{Reg}}}(t)$
 is square integrable. Equivalently, from \eqref{solnfilter}, square integrability translates to the requirement that for vector $\bo k$ with entries $\bo k_i=\sub{q_i} \sub{\beta_i}/\sub{\sigma_i}$, $i=1\dots n$, $\|\bo k\|_2^2<\infty$. 
 
 \subsection{Regularization parameter estimation}\label{sec:regest}
Without loss of generality we drop superscripts on the variables,  and  suppose for the moment that the system matrix $A$ in \eqref{discrete system}  is of size $n\times n$, and $g(s)$ is sampled at $n$ points.  We consider the solution of the Tikhonov regularized problem 
\begin{align}\label{regform}
\bo{ x}_{\text{Reg}}(\lambda)=  \argmin_{\bo x} \{ \|\bo R(\lambda)\|_2^2\}=: \argmin_{\bo x} \left\{\|{ A}\bo{x}-\bo{ b}\|_2^2+\lambda^2\|{\bo{x}}\|_2^2\right\},
\end{align}
where here $\bo R$ defines the residual for the approximate augmented system, $[A; \lambda I] \bo{x} \approx [\bo{b};\bo{0}]$. Consistent with the topic of this paper, we assume that the SVD is used to write the filtered solution  
\begin{align}\label{regsoln}
 \bo{ {x}}_{\text{Reg}}(\lambda)= \sum_{i=1}^n \frac{\sigma_i^2}{\lambda^2+\sigma_i^2}\frac{\bo u_i^T \bo b}{\sigma_i}  \bo v_i = \sum_{i=1}^n q(\lambda,\sigma_i)\frac{\beta_i}{\sigma_i}  \bo v_i, \quad q(\lambda,\sigma)=\frac{\sigma^2}{\lambda^2+\sigma^2 },
\end{align}
and  we make the observation $\| \bo{ {x}}_{\text{Reg}}(\lambda)\|_2^2= \|\bo k\|_2^2$. In the subsequent discussion of the methods for finding the regularization parameter we also use the SVD, in line with the overall theme of this paper which is focused on the use of the SVD and not on other techniques to find the solution.

Many methods exist for determining a regularization parameter $\lambda$ which will yield an acceptable solution with respect to some measure defining \textit{acceptable}. Amongst others, these   include the Morozov discrepancy principle (MDP) \cite{morozov}, the unbiased predictive risk estimator (UPRE) \cite{vogel:02}, generalized cross validation (GCV) \cite{golub1979generalized} and the $\chi^2$ principle applied for the augmented system (ADP) \cite{mere:09}.   The  MDP, UPRE and ADP  all assume prior   statistical information on the noise in the measurements $\bo b$. For simplicity in the presentation of the relevant functionals we assume that  $\bo b_{\text{obs}}=\bo b + \bo e$ where the noise is white with variance $\zeta^2$,  $\bo e \sim \Nn(0,  \zeta^2 I_n)$.  Then parameter $\lambda$ is found as the  minimum of a nonlinear functional for both the UPRE and GCV,  but the root of a monotonically increasing function for both the  MDP and ADP.  Defining residual $\bo r(\lambda) = A \bo x(\lambda)- \bo b$ and   influence matrix $A(\lambda) = A^T(A^TA+\lambda^2I)^{-1}A^T$,   the relevant functionals are given by 
\begin{align}\label{MDP}
\textrm{MDP}:\quad  &D(\lambda) = \|\bo r(\lambda)\|_2^2= \sum_{i=1}^n \left(1-q(\lambda,\sigma_i)\right)^2 \beta_i^2  =\zeta^2 \tau, \quad 0<\tau<n, \\ \label{ADP}
\textrm{ADP}:\quad  &C(\lambda)= \|\bo R(\lambda)\|_2^2 = \sum_{i=1}^n \left(1-q(\lambda,\sigma_i)\right)\beta_i^2    =\zeta^2 n, \\ \label{upre}
\textrm{UPRE}:\quad  &U(\lambda)=\|\bo r(\lambda)\|_2^2 +2 \zeta^2  \mathrm{trace}(A(\lambda))= \sum_{i=1}^n \left(1-q(\lambda,\sigma_i)\right)^2 \beta_i^2 + 2 \zeta^2  \sum_{i=1}^n q(\lambda,\sigma_i),\\ \label{gcv}
\textrm{GCV}:\quad  &G(\lambda)=\frac{n^2 \|\bo r(\lambda)\|_2^2}{ \left(\mathrm{trace}(I- A(\lambda))\right)^2}=\frac{n^2 \sum_{i=1}^n \left(1-q(\lambda,\sigma_i))\right)^2 \beta_i^2}{\left( n-\sum_{i=1}^n q(\lambda,\sigma_i)\right)^2}.
\end{align}
Here, the introduction of the weight $n^2$ in the numerator of $G$ is non standard but convenient, and does not impact the location of the minimum. For \eqref{MDP} we note that the specific derivation depends on the $\chi^2$ distribution of the residual, which,  for a square  invertible system,  has theoretically $0$ degrees of freedom. Heuristically, it is standard to use    $\tau=n$, for white noise with variance $1$, so that the solution fits the data on average to within one standard deviation \cite{ABT}. On the other hand, the ADP functional also arises from a statistical analysis but as stated here is  under the assumption that $\bo x_0=0$. Then $C(\lambda)/\zeta^2 \sim \chi^2(n)$, i.e. the weighted functional  follows a $\chi^2$ distribution with $n$ degrees of freedom, so that $E(C(\lambda))=n\zeta^2$. 

For completeness we also comment on the more general problem  e.g. \cite{Hansen,Regtools,vogel:02},    
\begin{align}\label{genregsoln}
\bo x_{\text{Reg}}(\lambda, L, W, \bo x_0)=\argmin_{\bo x} \{J(\bo x)\}=\argmin_{\bo x} \left\{\norm{A\bo{x}-\bo{b}}_W^2+\lambda^2\norm{L(\bo{x}-\bo x_0)}_2^2\right\}.
\end{align}
 When  $L$ is invertible we may solve with respect to  the transformation $\bo y = L \bo{x}$  for the system with matrix $AL^{-1}$, hence mapping the basis for the solution. Often, however, $L$ is chosen as a noninvertible smoothing operator, say approximating a derivative operator. This case is not considered here, requiring extensions of the generalized singular value decomposition (GSVD) instead of the SVD.  In \eqref{genregsoln} $\bo x_0$  constrains the solution to be close to $\bo x_0$. For the ADP it is assumed that $E(\bo x) = \bo x_0$, for the solution, $\bo x$, considered as a stochastic variable.   Assuming $W$ is positive definite,  with $\tilde{A}=W^{1/2}A$ and  $\tilde{\bo b}=W^{1/2}(\bo b - A \bo x_0)$,     
\begin{align*}
\bo{ \tilde{x}}_{\text{Reg}}(\lambda)=  \argmin_{\bo x} \left\{\|{\tilde{A}\bo{x}-\bo{\tilde{b}}}\|_2^2+\lambda^2\|{\bo{x}}\|_2^2\right\} + \bo x_0.
\end{align*}
When $W = \Cb^{-1}$ the noise in $\bo b$ is whitened and   \eqref{MDP}-\eqref{upre} apply with $\zeta^2=1$. We assume from here on that $L$  is diagonal and invertible,  and only consider \eqref{regform} in the remaining discussion concerning the numerical implementation. We note that  we do still need to assess whether this left and right preconditioning of the original system matrix  $A$, by $W^{1/2}$ on the left or $L^{-1}$ on the right, resp., has an impact on the  square integrability of the  modified kernel and will discuss this in section~\ref{mod kernel}.

Practically, it is useful  to truncate the expansion for  $\bo{  {x}}_{\text{Reg}}(\lambda)$ at the numerical rank of the matrix $A$, dependent on the machine precision, effectively removing from \eqref{regsoln} terms dependent on the smallest singular values, those which are not resolved due to the numerical precision of the software environment.  This is equivalent to taking the filter factors $q(\lambda, \sigma_i)=0$  for small $\sigma_i$. Typically, these coefficients would be filtered out by the appropriate choice of $\lambda$, but for the purposes of our analysis it is helpful to determine the $i$ for which the coefficient becomes zero with respect to the machine precision, and introduce the concept of numerical rank. 
\begin{definition}[Numerical Rank]\label{numericalrank}
We define the numerical rank with respect to a given precision $\epsilon$ to be $p=\{ \max{i} : \sigma_i > \epsilon\}$. 
\end{definition} 
Because of the ordering of the singular values,  and defining $q(\lambda,\sigma_i)=0$ for $i>p$ so that the last $n-p$ terms are removed, \eqref{regsoln} is replaced by the truncated solution 
\begin{align*}
 \bo{ {x}}_{\text{TSVDReg}}(\lambda)= \sum_{i=1}^p q(\lambda,\sigma_i)\frac{\bo u_i^T \bo b}{\sigma_i}  \bo v_i, 
\end{align*}
 e.g. \cite{Hansen,RHM:10}.  Consequently,    functionals \eqref{MDP}-\eqref{gcv} are modified. For example, in \eqref{ADP}  we obtain  
 \begin{align*}
C(\lambda)= \sum_{i=1}^p \left(1-q(\lambda,\sigma_i)\right)\beta_i^2     + \sum_{i=p+1}^n  \beta_i^2 .
 \end{align*}
 In this case the analysis still finds $E(C(\lambda)/\zeta^2)\sim\chi^2(n)$ \cite{RHM:10}, but now  $E(\sum_{i=p+1}^n  \beta_i^2) =\zeta^2 (n-p)$ for large enough $n-p$ \cite{mere:09}.  For the UPRE functional the truncation simply introduces constant terms which can thus be ignored in the minimization.  We obtain the truncated expressions
 \begin{align}\label{TMDP}
D_{\text{T}}(\lambda) &= \sum_{i=1}^p \left(1-q(\lambda,\sigma_i)\right)^2 \beta_i^2  =\tau \zeta^2 , \quad 0<\tau<p, \\ \label{TADP}
C_{\text{T}}(\lambda)&=  \sum_{i=1}^p \left(1-q(\lambda,\sigma_i)\right)\beta_i^2   =\zeta^2 p, \\  \label{Tupre}
U_{\text{T}}(\lambda)&=\sum_{i=1}^p \left(1-q(\lambda,\sigma_i)\right)^2 \beta_i^2 + 2 \zeta^2 \sum_{i=1}^p q(\lambda,\sigma_i), \\ \label{Tgcv}
G_{\text{T}}(\lambda)&=\frac{n^2 \left(\sum_{i=1}^p \left(1-q(\lambda,\sigma_i)\right)^2 \beta_i^2 +\sum_{i=p+1}^n \beta_i^2\right)}{\left(n-p+ \sum_{i=1}^p q(\lambda,\sigma_i)\right)^2}.
\end{align}
Observe that the GCV apparently needs the coefficients $ {\beta_i}$ for all $i$, regardless of choice of $p$. On the other hand for orthogonal $U$ and given $\bo b$ of length $n$ we note 
\begin{align*}  
\|\bo b\|_2^2 =\|U^T\bo b\|_2^2 = \sum_{i=1}^p \beta^2_i + \sum_{i=p+1}^n \beta_i^2\quad \text{yields}  
\sum_{i=p+1}^n \beta_i^2  = \|\bo b\|^2 -  \sum_{i=1}^p \beta^2_i.
\end{align*}
Thus indeed the GCV can be evaluated  without a complete SVD for the system of size $n$. 

As to the choice of the analysis of these methods, and exclusion of other regularization parameter selection techniques, we note that there are numerous techniques that can be applied.  We did not choose to discuss  the well-known L-curve, e.g. \cite{HansenLC}.  There the techniques for analysis are somewhat different, requiring the analysis of the curvature of the L-curve. Vogel \cite{vogel:02} did  consider the L-curve, however with less positive results. In particular, he found that the L-curve either becomes flat with increasing $n$, or gives a value that does not lead to mean square convergence of the error. Thus here we have chosen to consider the particular selection methods   and ignore for now the L-curve. 
To apply  regularization parameter selection techniques practically  in the context of the SVE-SVD relation,  we first examine the determination of the numerical rank $p$ for the  set of system matrices  $\{\sub{A}\}$ with increasing $n$ and the square integrability of the kernel $H(s,t)$ under the variable mappings that  correspond to the left and right preconditioning of  the matrix $A$.

\section{Theoretical Results}\label{theoryresults}
\subsection{Square integrability of the weighted  kernel}\label{mod kernel}
The theoretical justification for using an estimate of the regularization parameter $\lambda$ obtained from a down sampled set of data, and appropriately down sampled kernel matrix, relies primarily on the results on  the SVE-SVD relationship discussed in Theorem~\ref{SVE-SVD}.  But as noted for \eqref{genregsoln} it is important to discuss the impact of the left and right preconditioning of the matrix $A$ which results from replacing $A$ by $A=W^{1/2} A L^{-1}$ in \eqref{regform}.  For diagonal matrices $W$ and $L$ it is immediate that premultiplication by $W^{1/2}$ amounts to a row scaling and post multiplication by $L^{-1}$ to a column scaling. Consistent with the practical data we suppose that $W \approx C^{-1}$ for symmetric positive definite (SPD) matrix $C$. Then  $C^{-1/2}$ is a sampling of a function $c(s)\ne 0$, by $C$  SPD,  and $L$ the sampling of a function $\ell(t)\ne 0$, by the invertibility of $L$. The kernel is replaced by a weighted rational kernel $\tilde{H}(s,t)= H(s,t)/(c(s)\ell(t))$. 
\begin{theorem}\label{thm:weighted kernel}
For bounded  functions $c(s)>c_0>0$ and $|\ell(t)|>\ell_0> 0$  defined on $\Omega_s$ and $\Omega_t$ respectively, the square integrability of the weighted kernel $\tilde{H}(s,t)=H(s,t)/(c(s)\ell(t))$ defined on $\Omega_s\times\Omega_t$ follows from the square integrability of $H$ on the same domain. 
\end{theorem}
\begin{proof}
The proof is immediate from
\begin{align*}
\|\tilde{H}\| = \| \frac{1}{c(s)} H(s,t) \frac{1}{\ell(t)}\| \le \frac{1}{c_0 \ell_0} \|H\| <\infty.
\end{align*}
\end{proof}

We have thus determined that we may use the relation between the SVE and the SVD for the      left and right preconditioned kernel, $\tilde{H}$. Furthermore, data $g(s)$ and source function $f(t)$ are mapped accordingly without impacting the analysis, i.e. we have the mapped source $\tilde{f}(t) = \ell(t) f(t) $ and mapped data $\tilde{g}=g(s)/c(s)$.  
\subsection{Convergence of the SVD to the SVE and Numerical Rank}\label{theory rank}
Although Theorem~\ref{SVE-SVD} summarizes the primary results used in our analysis, some additional results are useful for further analysis with respect to the numerical rank.  First we focus on $\sub{\Delta}$ defined in \eqref{defDelta} which provides an estimate of the error in the estimation of $H$ using $\sub{A}$.  If $\lim\limits_{n\to\infty}{\sub{\Delta}}=0$, then matrix $\sub{A}$ effectively becomes independent of its discretization in providing an accurate representation of the integral with kernel $H$.   
\begin{prop}\label{prop1}Suppose ONB $\{\phi_j\}_{j=1}^\infty$ and $\{\psi_i\}_{i=1}^\infty$  in Algorithm~\ref{MM}   are complete, then 
$\lim\limits_{n\to\infty} \left(\sub{\Delta}\right)^2 = 0$.
\end{prop}
\begin{proof}
$H(s,t)$ is defined on $\Omega_s\times\Omega_t$, i.e. $H\in L^2(\Omega_s\times\Omega_t)$.  Let $\{\phi_j(t)\}$ be an ONB for $L^2(\Omega_t)$ and $\{\psi_i(s)\}$ be an ONB for $L^2(\Omega_s)$. Then $\{\phi_j(t)\psi_i(s) \, \big| \, i,j\in\Z^+\}$ is an ONB  for $L^2(\Omega_s\times\Omega_t)$. Setting, 
 \begin{align*}
 H(s,t) &= \sum\limits_{i=1}^\infty\sum\limits_{j=1}^\infty a_{ij}\phi_j(t)\psi_i(s)  \quad \text{yields}\\ 
 \lra{\lra{H,\psi_i},\phi_j} &=  \dis{\int_{\Omega_t}\int_{\Omega_s} \psi_i(s)H(s,t)\phi_j(t) \, ds \, dt }= a_{ij}.
 \end{align*} 
Immediately, $\norm{H(s,t)}^2=\sum\limits_{i=1}^\infty\sum\limits_{j=1}^\infty\abs{a_{ij}}^2$, and defining $\sub{A}=[a_{ij}]_{i,j=1}^n$ yields 
 \begin{align*}\left(\sub{\Delta}\right)^2 =\|H\|^2 -\|\sub{A}\|_F^2&= \sum\limits_{i=1}^\infty\sum\limits_{j=1}^\infty\abs{a_{ij}}^2 - \sum\limits_{i=1}^n\sum\limits_{j=1}^n\abs{a_{ij}}^2 =  \sum\limits_{\max(i,j)>n} \abs{a_{ij}}^2. 
 \end{align*} 
But now by square integrability $\norm{H}^2<\infty$, so the expression is a convergent series. Thus, by Cauchy's criterion the tail end of the sum converges to zero with $n$.
\end{proof}
Practically we suppose the discrete system  in \eqref{discrete system} is constructed using different basis functions for each choice of $n$, and that $\sub{a_{ij}}$ are calculated using a quadrature rule. Then Proposition~\ref{prop1} may  not immediately apply due to quadrature error. But, for the theory,    we assume $\left(\sub{\Delta}\right)^2\xrightarrow{n \to \infty} 0$. We also assume that all continuous singular values are distinct,  as required for Theorem~\ref{SVE-SVD} statement~\ref{thm4}. Analogues of this result hold for the case without distinct singular values \cite{hansensvepaper}. 

In relating the numerical rank $\sub{p}$ across resolutions it is helpful to 
 establish a few basic results, the first two of which effectively appear in \cite{hansensvepaper}. 
\begin{lemma} \label{lemma1}
If $\left(\sub{\Delta}\right)^2\xrightarrow{n \to \infty} 0$, then $\lim\limits_{n\to\infty}{\sub{\sigma_i}}=\mu_i$ for all $i$.
\end{lemma}
\begin{proof}
By Theorem~\ref{SVE-SVD} statements~\ref{thm1} and \ref{thm2}, 
it is immediate that $\lim\limits_{n\to\infty}{\sub{\sigma_i}}=\mu_i$  for all $i$.
\end{proof}
\begin{lemma} \label{lemma3}
If $\left(\sub{\Delta}\right)^2\xrightarrow{n \to \infty} 0$  then
$\lim\limits_{n\to\infty} {\sub{\beta_i}} =\lra{u_i, g} = \hat{g}_i$  for all $i$.
\end{lemma}
\begin{proof}
By Theorem~\ref{SVE-SVD} statement~\ref{thm4}, %
$\lim\limits_{n\to\infty}{\sub{\tilde{u}_i}}=u_i$.   Using \eqref{betacoeff} and recalling  $\sub{\beta_i} =(\sub{\bo u_i})^T\bo b$, yields  $\sub{\beta_i}=\lra{\sub{\tilde{u}_i}, g}$.  Thus  $ \lim\limits_{n\to\infty} {\sub{\beta_i}} = \lim\limits_{n\to\infty}\lra{\sub{\tilde{u}_i}, g}=\lra{u_i, g} = \hat{g}_i$.
\end{proof}
\begin{lemma} \label{lemma4}
For all $i$ and $n,$ $\sub{\sigma_i} \le \mu_i \le {\sub{\Delta}} + {\sub{\sigma_{i-1}}}$ and $\mu_{i+1}-{\sub{\Delta}} \le {\sub{\sigma_i}} \le \mu_i$.
\end{lemma}
\begin{proof}
From Theorem~\ref{SVE-SVD} statement~\ref{thm2},  and using the ordering of  $\{\sub{\sigma_i}\}_{i=1}^n$,  $\mu_i-{\sub{\sigma_i}}\le {\sub{\Delta}}$,   implies $\mu_i \le {\sub{\Delta}} + {\sub{\sigma_i}} \le {\sub{\Delta}} + {\sub{\sigma_{i-1}}}$. 
By Theorem~\ref{SVE-SVD} statement~\ref{thm1}  this gives  $\sub{\sigma_i} \le \mu_i \le {\sub{\Delta}} + {\sub{\sigma_{i-1}}}$.  Reindexing and subtracting $\sub{\Delta}$ also provides $\mu_{i+1}-{\sub{\Delta}} \le {\sub{\sigma_i}} \le \mu_i$.
\end{proof}
To relate the convergence between continuous and discrete spectra, we introduce  the parameter $\epsilon$ which arises in the Definition \ref{numericalrank} of the numerical rank and   depends  on  the machine precision.  
\begin{theorem}[Numerical Rank] \label{lemma2}
Let us assume that $\left(\sub{\Delta}\right)^2\xrightarrow{n \to \infty} 0$ and that all continuous singular values $\mu_i$ are distinct. Let $P\in\Z^+ \st \mu_P>\epsilon$ and $\mu_{P+1}\le \epsilon$, for small positive $\epsilon$, then $\lim\limits_{n\to\infty} {\sub{p}} = P =: p^*$.  Moreover, there exists $n^*\in\Z^+\st {\sub{p}}=p^* $ for all $ n\ge n^*$.
\end{theorem}
\begin{proof}
First note that because $\norm{H}^2 = \sum\limits_{i=1}^\infty\mu_i^2<\infty$,  $\mu_i\to0$,   and $P$   exists.  From Theorem~\ref{SVE-SVD} statement~\ref{thm1},  $\sub{\sigma_{p^*}}\le{\sigma^{(n+1)}_{p^*}}\le\mu_{p^*}$.  Additionally, for all $ i>p^*, \, {\sub{\sigma_i}}\le\mu_{p^*+1}<\epsilon$.  Thus $\sub{p}\le p^*$. %
 From Lemma~\ref{lemma4}, $\epsilon <\mu_{p^*} \le {\sub{\Delta}}+ \sub{\sigma_{p^*-1}} \le {\sub{\Delta}}+ \sub{\sigma_{i}}$, for any $i<p^*-1$.  Thus as $\sub{\Delta}\to0$,  $\sub{\sigma_i}> \epsilon$, for  all $i\le p^*-1$, yielding $\sub{p}\ge p^*-1$, i.e. $p^*-1\le  \sub{p}\le p^*$. Suppose $\sub{p} = p^*-1$, then   $\sub{\sigma_{p^*}}<\epsilon <\sub{\sigma_{p^*-1}}$. But by Theorem~\ref{SVE-SVD} statement~\ref{thm1},  $0\le \mu_{p^*} -\sigma_{p^*} \le \sub{\Delta}$, and  as $\sub{\Delta}\to0$, $\sigma_{p^*}\to\mu_{p^*}>\epsilon$. Hence  $\lim\limits_{n\to\infty} {\sub{p}} = p^*$, and there exists $ n^*\in\Z^+\st {\sub{p}}=p^*$ for all $n\ge n^*$.  
\end{proof}

\subsection{Convergence of $\sub{\lambda}$}\label{sec:lambdaconv}
We define  regularization parameter $\sub{\lambda}$ to be the estimate of the regularization parameter with resolution $n$, and $\lambda^*$ to be the estimate for regularizing the continuous solution \eqref{contsoln}. Now, to assist in the analysis we introduce some notation for functionals that occur repeatedly in the formulation. Specifically, the  regularization functionals \eqref{TMDP}-\eqref{Tgcv}  are expressible in terms of the common multivariable function 
\begin{align*}
\eta(\lambda, p, k, \bo a, \bo z) &= \sum_{i=1}^{p} z^2_i \paren{\frac{\lambda^2}{a_i^2+\lambda^2}}^k =\sum_{i=1}^{p}{z^2_i} \paren{1-{q}(\lambda, a_i)}^k\\
&=\left({\bo z^2}\right)^T ({\bfone} -  {\bo w} (\lambda,  {\bo a}))^k \end{align*}
where we have defined the  vectors  $\bo z $, $\bfone$, $\bo w$ and $\bo a$  $\in \mathcal{R}^p$,  with  $w_i =  q(\lambda, a_i)$ and $\bfone_i$ $=1$.  With the appropriate identification of the terms in $\eta$ we obtain, 
\begin{align*}
\sub{D_{\text{T}}}(\lambda) &=\eta(\lambda,\sub{p}, 2, \sub{\boldsymbol{\sigma}}(1:\sub{p}), \sub{\boldsymbol{\beta}}(1:\sub{p}))  \\ 
\sub{C_{\text{T}}}(\lambda)&=  \eta(\lambda,\sub{p}, 1, \sub{\boldsymbol{\sigma}}(1:\sub{p}), \sub{\boldsymbol{\beta}}(1:\sub{p})) \\ 
\sub{U_{\text{T}}}(\lambda)&=\eta(\lambda,\sub{p}, 2, \sub{\boldsymbol{\sigma}}(1:\sub{p}), \sub{\boldsymbol{\beta}}(1:\sub{p})) + 2 \zeta^2   \bfone^T\bo w(\lambda, \sub{\boldsymbol{\sigma}}(1:\sub{p}))\\
\sub{G_{\text{T}}}(\lambda)&=  \frac{ n^2 \left(\eta(\lambda,\sub{p}, 2, \sub{\boldsymbol{\sigma}}(1:\sub{p}), \sub{\boldsymbol{\beta}}(1:\sub{p}))  +\|\boldsymbol{\beta}(\sub{p}+1:n)\|_2^2\right)}{\left(n-\sub{p}+ \bfone^T\bo w(\lambda, \sub{\boldsymbol{\sigma}}(1:\sub{p})) \right)^2}. \end{align*}
Equivalent continuous functionals  are obtained for  effective continuous rank $p^*$ as in Theorem~\ref{lemma2},  by  defining $z_i=\hat{g}_i$, and $a_i=\mu_i$. For example,   the limiting GCV  is
\begin{align}
\label{contgcv}
G^*(\lambda)&=\lim_{n \to \infty}\left\{ \frac{n^2 \left(\eta(\lambda,p^*,2,\bfmu(1:p^*),   \hat{\bo g}(1:p^*) )  +\|\boldsymbol{\beta}({p}^*+1:n)\|_2^2\right)}{\left(n-p^*+    \bfone^T\bo w(\lambda, \bfmu(1:p^*))\right)^2} \right\}.
\end{align}

To relate the continuous and discrete functionals,   note immediately the continuity of $\bo w$ with respect to $\bo a$ and $\lambda$,  and hence of $\eta$ with respect to $\bo a$, $\bo z$ and $\lambda$.  Moreover, using  Lemma~\ref{lemma1} and Theorem~\ref{SVE-SVD} statement~\ref{thm1} for $\mu_i$, and  Lemma~\ref{lemma3} for $\hat{g}_i$,   we introduce 
\begin{align}\label{expsigma}
\subs{\sigma}&=\mu_i+{\subs{\epsilon}}, \quad \subs{\epsilon}<0, \quad \lim\limits_{n\to\infty}{\subs{\epsilon}}=0\\ \label{expbeta}
 \sub{\beta_i}&=\hat{g}_i+{\subs{\delta}}, \quad \lim\limits_{n\to\infty}{\subs{\delta}}=0.
 \end{align}
 \begin{lemma}[Convergence of $\eta$]\label{lemmaetaconv}
 Suppose $n>n^*$ such that Theorem~\ref{lemma2} holds,  and in  \eqref{expsigma} and \eqref{expbeta} $|\subs{\epsilon}|< \subs{\sigma}$, and $|\subs{\delta}|< |\hat{g}_i|$, respectively. Then, 
  \begin{align}
 \nonumber 
\lim_{n\to\infty} \eta(\lambda,{\sub{p}}, k, \sub{\boldsymbol{\sigma}}(1:\sub{p}), \sub{\boldsymbol{\beta}}(1:\sub{p})) &=   \eta(\lambda,p^*,k, \bfmu(1:p^*),   \hat{\bo g}(1:p^*)) \\
\lim_{n \to \infty}   \bfone^T\bo w(\lambda, \sub{\boldsymbol{\sigma}}(1:\sub{p})) &=\bfone^T \bo w(\lambda, \bfmu(1:p^*))\nonumber\\
 \lim_{n\to\infty} \frac{n^2}{\left(n-\sub{p}+    \bfone^T\bo w(1:\sub{p})\right)^2} &=1, \quad p^*<< n. \nonumber
 \end{align}
 \end{lemma}
 \begin{proof} The proof is immediate by the continuity of $\eta$  and $\bo w$. 
 \end{proof}
 
 For ease of notation we introduce $\bar{\bo z}$ to be vector $\bo z$ truncated to length ${p}^*$.
 \begin{theorem}[Convergence of $\sub{\lambda}$]\label{regconv}
  Suppose $n>n^*$ such that Theorem~\ref{lemma2} holds,  and in  \eqref{expsigma} and \eqref{expbeta} $|\subs{\epsilon}|< \subs{\sigma}$, and $|\subs{\delta}|< |\hat{g}_i|$, respectively. Assume $\sub{\lambda}$ and $\lambda^*$ are given by one of the following cases: 
 \begin{enumerate}
 \item $\sub{\lambda}$ solves  $\sub{D_{\text{T}}}(\lambda) = \zeta^2 \tau$ and $\lambda^*$ solves $D^*(\lambda)  =\eta(\lambda,p^*,2,  \bar{\bfmu}, \bar{\hat{\bo g}}) = \zeta^2 \tau$. 
 \item $\sub{\lambda}$ solves  $\sub{C_{\text{T}}}(\lambda) = \zeta^2 \sub{p}$ and $\lambda^*$ solves $C^*(\lambda)  =\eta(\lambda,p^*,1,  \bar{\bfmu}, \bar{\hat{\bo g}}) = \zeta^2 p^*$.  
 \item $\sub{\lambda} = \argmin_{\lambda} \sub{U_{\text{T}}}(\lambda)$ and $\lambda^* = \argmin_{\lambda}U^*(\lambda)  =\argmin_{\lambda} \{\eta(\lambda,p^*,2,  \bar{\bfmu}, \bar{\hat{\bo g}}) +2 \zeta^2   \bfone^T\bo w(\lambda, \bar{\bfmu}) \}$.  
 \item $\sub{\lambda} = \argmin_{\lambda} \sub{G_{\text{T}}}(\lambda)$ and $\lambda^* = \argmin_{\lambda}G^*(\lambda)$, $G^*(\lambda)$ as defined in \eqref{contgcv}. 
  \end{enumerate}
 Then, in each case,  $\lim\limits_{n\to\infty}{\sub{\lambda}}=\lambda^*$.
 \end{theorem}
 \begin{proof}
 The result follows by Lemma~\ref{lemmaetaconv} immediately for the MDP, ADP and UPRE functionals.  For the GCV we note in addition  that  $\lim_{i \to \infty} (\sub{\beta_i})^2=0$. Therefore $\lim_{n\to \infty} \|\boldsymbol{\beta}(p^*+1:n)\|_2^2$ is bounded and independent of $\lambda$.  
 \end{proof}

\section{Practical Implementation}\label{sec:practical}
Our interest, as noted, is the solution of the integral equation \eqref{eq:fred} rather than the generation of an approximation to the SVE for the kernel $H(s,t)$. We carefully describe the stages of the algorithm which lead to the determination of the solution of the large scale problem, using the regularization parameter estimated using only the coarse resolution system of equations.  In the following we generally assume that  the data $g(s)$ is provided at a discrete set of points, $\{\subl{s_\imath}\}_{\imath=1}^N $, cf. \cite{hansensvepaper}, and is contaminated by noise vector $\bo e$, where $\bo e \sim \Nn(0,  C)$ for diagonal matrix $C=\mathrm{diag}(\zeta_1^2, \zeta_2^2, \dots, \zeta_N^2)$. 
We also assume that the approximation of $f(t)$ is required at $N$ points $\{\subl{t_\jmath}\}_{\jmath=1}^N $.
\subsection{The downsampled system}\label{sec:downsample}
To apply the algorithm with respect to different resolutions we have to first identify  the  sampling, i.e. we  pick the coarse level resolution $n<N$ such that it is possible to find a sampling $\{\sub{s_i}\}_{i=1}^n $,  with $\sub{s_i}=\subl{s_\imath}$ for some $\imath \in \iota$ for index set $\iota$.  For example if  $n=N/2$ then we may take every second sampled point at the fine resolution for the downsampled data so that $\iota=\{1, 3, \dots, N-1\}$. The samples are ordered $\subl{s_1}\le \sub{s_1}< \sub{s_2} \dots <\sub{s_n}\le \subl{s_N}$ and yield the sampling vector with entries $\subs{g} = g(\subs{s})$. 

We now use the   indicator  functions, normalized to length $1$,   given by
\begin{align}\label{indicator}
\chi_k(x) = \left\{\begin{array}{ll} \frac{1}{\sqrt{ dx_k}}  & ~~~x \in \Omega_k =[x_k-\frac{dx_k}{2}, x_{k}+\frac{dx_k}{2}]\\ 0 &~~~\textrm{otherwise}\end{array} \right. .
\end{align}
Then,  defining  step size $\sub{d s_i}=\sub{s_{i+1}}-\sub{s_i}$ with the equivalent definitions in $t$, and such that the sample points are at the mid points of each non-overlapping interval, 
\begin{align}\label{ONB}
\sub{\psi_i}(s)    =  \left\{\begin{array}{ll} \frac{1}{\sqrt{\sub{d s_i}}}  &  s \in \Omega_{\sub{s_i}} =[\sub{s_i}-\frac{\sub{d s_i}}{2}, \sub{s_i}+\frac{\sub{d s_i}}{2}]\\ 0 & \textrm{otherwise}\end{array} \right..  
\end{align}
Thus, in \eqref{discrete system}, with the assumption that the integral over $\Omega_{\sub{s_i}}$ uses the mid point rule, we obtain the approximation
\begin{align}\subs{\hat{g}} = \langle{g(s), \psi_i(s)}\rangle \approx g(s^{(n)}_i) \sqrt{\sub{d s_i}}=:\sub{b_i}. \label{coeffgnonoise}\end{align}
The function $f(t)$ is defined similarly, for indicator basis functions $\sub{\phi_j}(t)$,  as in \eqref{indicator}, so that 
\begin{align} \subj{\hat{f}} &= \langle{f(t), \phi_j(t)}\rangle \approx f(t^{(n)}_j) \sqrt{\sub{d t_j}}=:\sub{x_j}, \nonumber \\ 
  \text{where}  \quad \sub{\phi_j}(t)   &=  \left\{\begin{array}{ll} \frac{1}{\sqrt{\sub{d t_j}}}  &  t \in \Omega_{\sub{t_j}} =[\sub{t_j}-\frac{\sub{d t_j}}{2}, \sub{t_j}+\frac{\sub{d t_j}}{2}]\\ 0 &  \textrm{otherwise}\end{array} \right.. \label{ONB2}
 \end{align} 
Then, again assuming the mid point rule, the approximation to the kernel matrix is given by  
\begin{align}\nonumber
\int_{\Omega_{{s}}}\int_{\Omega_{{t}}}\sub{\psi_i}(s)  H(s,t) \sub{\phi_j}(t)  dt ds&=\frac{1}{\sqrt{\sub{d s_i} \sub{d t_j}}}\int_{\Omega_{\sub{s_i}}}\int_{\Omega_{\sub{t_j}}}H(s,t) dt ds\\ &\approx  {\sqrt{\sub{d s_i} \sub{d t_j}}} H(\subs{s}, \sub{t_j}) =: \sub{a_{ij}}.
 \label{Anmatrix}\end{align}

For the resolution-based algorithm   $\subl{A}$  is required and can be also calculated using  \eqref{Anmatrix}. Thus $\sub{A}$ can be obtained by sampling and scaling  $\subl{A}$, i.e. by extracting rows and columns with the correct scaling, 
\begin{align}
\sub{a_{ij}}  =\sqrt{\sub{d s_i} \sub{d t_j}} H(\subs{s}, \sub{t_j}) 
  =\frac{\sqrt{\sub{d s_i} \sub{d t_j}}}{\sqrt{\subl{d s_\imath} \subl{d t_\jmath}}} \subl{a_{\imath\jmath}}. \label{scaleA}
\end{align}
We note that the impact of the quadrature error in the calculation of these elements, which tends to $0$ with $n$,  is ignored in the analysis. 

We have demonstrated, therefore, that the computational cost for determining the matrix for the coarse grain resolution $\sub{A}$ is negligible compared to the cost for determining the kernel matrix $\subl{A}$ which would be required independent of any coarse-fine resolution arguments.\footnote{If the kernel integral is calculated exactly over the given interval it is still possible to obtain $\sub{A}$ from $\subl{A}$ by summing the relevant terms from $\subl{A}$ but the scaling factor is the inverse of that in \eqref{scaleA}.} Further,  in obtaining the sampling $\{\sub{s_i}\} $, it is appropriate to define a sampling interval $\ell$ such that   $\imath = 1 : \ell:N$, yielding non-overlapping intervals. 
When the sampling is completely uniform, with all of $\subl{d s_\imath}$, $\subl{d t_\imath}$, $\sub{d s_i}$, and $\sub{d t_i}$, independent of index $\imath$ and $i$, $\sub{A} =\alpha \subl{A}$, where $\alpha$  depends only on the ratios between the number of points at each resolution.

Given matrices $\sub{A}$ and $\subl{A}$, the goal is now to determine the numerical rank $\subl{p}$ and regularization parameter $\subl{\lambda}$ only using the SVD of $\sub{A}$. Then, using the first $\subl{p}$ terms of the SVD for $\subl{A}$, calculated for example using for example \textrm{svds}$(\subl{A}, \subl{p})$ in Matlab, and thus not requiring the full SVD for $\subl{A}$,   approximations at the original fine resolution are given by
  \begin{align}\nonumber 
 \subl{f}(t_k) &\approx \sum_{\jmath=1}^N  \left(\sum_{\imath=1}^\subl{p} q(\subl{\lambda},\subl{\sigma_\imath})\frac{(\subl{\bo u_\imath})^T \subl{\bo b}}{\subl{\sigma_\imath}}  \subl{ v_{\jmath\imath}}\right)\phi_\jmath(t_k) \\
 &=\frac{1}{\sqrt{\subl{d t_k}}}\sum_{\imath=1}^\subl{p} q(\subl{\lambda},\subl{\sigma_\imath})\frac{(\subl{\bo u_\imath})^T \subl{\bo b}}{\subl{\sigma_\imath}}  \subl{ v_{k\imath}}=: \subl{\bo f_k}. \label{tsvdregsolnN}\end{align}

\subsection{Determination of the numerical rank}\label{sec:numericalrank}
Theorem~\ref{lemma2}  suggests that the numerical rank estimation  introduces an additional significant parameter $\epsilon$ for the determination of $\sub{p}$. First we note that this parameter is not the  machine precision for floating point arithmetic $\eps_{\mathrm{float}}$ e.g. $ 2.2204e-16$ in  Matlab 2014b. While  $\epsilon$ does depend on $\epsilon_{\mathrm{float}}$  it also depends on the numerical spectrum of $\subl{A}$ and is easily estimated using the singular values for $\sub{A}$. In particular, $\epsilon$ is only relevant in determining the effective numerical rank of the problem, so as to determine the truncation of the SVD. We illustrate this for the numerical examples in section~\ref{simulations}, demonstrating the convergence of the singular values and consequent estimation of $\sub{p}$.

\subsection{Determination of the regularization parameter}\label{sec:regparam}
It remains to more carefully consider the estimation of the regularization parameter given the provided data. Suppose, for now,   that the measured data has been whitened via weighting of the sample vector  yielding $\tilde{\bo g} =:\zeta W^{1/2}\bo g$, where $W$ is the inverse covariance matrix for the noise vector $\bo e$, and $\zeta^2$ is the mean of the variance in each measurement. Likewise, then, the kernel matrix $\subl{A}$ is replaced by $\tilde{\subl{A}}=\zeta W^{1/2} \subl{A}$. We maintain the parameter $\zeta$ in the analysis just to emphasize its usage in the formulae in section~\ref{sec:regest} for estimating the regularization parameter. The weighted noise vector $\tilde{\bo e} = \zeta W^{1/2} {\bo e}   \sim \Nn(0,  \zeta^2I)$ and the estimation formulae apply as given in \eqref{TMDP}, \eqref{TADP}, and  \eqref{Tupre},  noting again by Theorem~\ref{thm:weighted kernel} that the scaling does not impact the convergence of the weighted singular values and coefficients. 
But now, again applying \eqref{coeffgnonoise}, the actual right hand side data are obtained by inner products with the weighted data
\begin{align*}
(\sub{\hat{\tilde{g}}_{\mathrm{obs}}}) _i=\langle{\tilde{g}_{\mathrm{obs}}(s),  \psi_i(s)}\rangle &= \langle{\tilde{g}(s)+\tilde{e}(s), \psi_i(s)} \rangle=\subs{ \tilde{b}}+ \subs{\hat{\tilde{e}}} \quad \text{for which}\\
\subs{\hat{\tilde{e}}} &= \tilde{e}(s_i) \sqrt{\sub{ds}}, \quad\subs{ \hat{\tilde{e}}}  \sim \mathcal{N}(0, \left(\zeta^{(n)}\right)^2 ), \quad \left(\zeta^{(n)}\right)^2 =\sub{ds}\zeta^2.
\end{align*}
Hence the noise level now depends on $n$ and we cannot immediately apply the convergence results, Theorem~\ref{regconv}, for \eqref{TMDP}-\eqref{Tupre} because the variance of the noise in each system depends on  $n$.  We  recall again that the GCV estimator is independent of $\zeta^2$ and no further discussion is needed.

It is immediate  by \eqref{regform} that if we introduce scaling of the kernel matrix and the right hand side data by constant $\mu$ then
\begin{align*}
\bo{ x}_{\text{Reg}}(\tilde{\lambda},\mu)=  \argmin_{\bo x} \left\{{\mu^2}\|{ A}\bo{x}-\bo{ b}\|_2^2+\tilde{\lambda}^2\|{\bo{x}}\|_2^2\right\} 
=\bo{ x}_{\text{Reg}}({\lambda}), \quad \lambda=\frac{\tilde{\lambda}}{\mu}.
\end{align*}
On the other hand, suppose that we find the regularization parameter $\sub{\tilde{\lambda}}$ using variance $\left(\zeta^{(m)}\right)^2$ instead of $\left(\sub{\zeta}\right)^2$ then this corresponds to scaling the data by $\mu=\zeta^{(m)}/\zeta^{(n)}$ yielding a solution  with regularization parameter 
\begin{align}\label{scaling}
{\lambda^{(n)}}={\frac{\zeta^{(n)}}{\zeta^{(m)}}} \tilde{\lambda}^{(n)}.
\end{align} 
We note that this result also follows by considering the expansion for the solution \eqref{regsoln}, and by the uniqueness of the SVD  \cite{GoLo:96}, at least with respect to the singular values and of the  singular vectors up to their signs.
Thus to apply the convergence results for \eqref{TMDP}-\eqref{Tupre} we calculate $\sub{\tilde{\lambda}}$ using $\left(\zeta^{(n)}\right)^2$ yielding $\subl{\lambda}=\left(\zeta^{(N)}/\zeta^{(n)} \right)\subl{\tilde{\lambda}}=\left(\zeta^{(N)}/\zeta^{(n)} \right)\sub{\tilde{\lambda}}$ by the convergence of the funcationals for constant variance. Then we may   use $\subl{\lambda}$, $\subl{\sigma_i}$ and $\subl{\beta_i}$ to find the solution using \eqref{tsvdregsolnN}. 
 For clarification the steps are described in Algorithm~\ref{alg}.

 \begin{algorithm}[H]\caption{Galerkin Method  to obtain regularized solution of  \eqref{eq:fred} \label{alg}}
\begin{algorithmic}[1]  
\Require  whitened   data $\{g(s_\imath)\}_{\imath=1}^N$, whitened kernel function $H(s,t)$ and  precision $\epsilon$. 
 \State Pick the coarse level resolution $n<N$ such that it is possible to find a sampling $\{g(\sub{s_i}\}_{i=1}^n)$,  with $\sub{s_i}=\subl{s_\imath}$, $\imath \in \iota$, for some index set $\iota$, with ordering $\subl{s_1}\le \sub{s_1}< \sub{s_2} \dots <\sub{s_n}\le \subl{s_N}$. 
 \State Choose ONB  $\{\sub{\phi_j}(t)\}_{j=1}^n$,  $\{\sub{\psi_i}(s)\}_{i=1}^n$,  $\{\subl{\phi_\jmath}(t)\}_{\jmath=1}^N$ and  $\{\subl{\psi_\imath}(s)\}_{\imath=1}^N$, e.g. using \eqref{ONB} and \eqref{ONB2}.
 \State  Calculate the right hand side vector $\sub{\bo b}$ using \eqref{coeffgnonoise}.
 \State   Calculate    $\subl{A}$ with entries $(\subl{a_{ij}})$ using \eqref{Anmatrix}. 
 \State  Calculate matrix $\sub{A}$ by sampling from $\subl{A}$ using \eqref{scaleA}.
 \State Compute  SVD, $\sub{A} =\sub{U} \sub{\Sigma} (\sub{V})^T$.    Estimate $\sub{p}$ using  $\epsilon$. 
 \State    Compute $\subl{p}=\sub{p}$ dominant terms for SVD of $\subl{A}$. 
 \State  Find $\sub{\tilde{\lambda}}$ using regularization parameter estimation by one of MDP, ADP, UPRE or GCV, using $\left(\zeta^{(n)}\right)^2$.
\State Relate $\sub{\tilde{\lambda}}$ to $\subl{\lambda}$, via $\subl{\lambda}=\sub{\tilde{\lambda}}\sqrt{\zeta^2 \sub{d s}}/\zeta^{(n)} = \sub{\tilde{\lambda}}\sqrt{\subl{d s}/\sub{ds}}$
\Ensure   the regularized  solution   \eqref{tsvdregsolnN} for resolution $N$ using $\subl{\boldsymbol{\beta}}$ and $\subl{\boldsymbol{\sigma}}$.
\end{algorithmic}
\end{algorithm}


Some advantages of the coarse to fine resolution argument are apparent. For example, suppose that   $\subl{{\lambda}}$ is found directly  by the UPRE. When $N$ is large, $\sub{ds}$ is small,  the noise in the coefficients goes to zero, forcing $\zeta^2=0$,  so that in the minimization the residual function dominates the filter terms.  For the ADP and MDP the right hand side also tends to zero with $\sub{ds} \to 0$, forcing the filter terms identically to $1$, i.e. to no filtering, which is consistent with the noise in the expansion coefficients going to $0$.  Noise due to the data sampling is effectively ignored at the high resolution, but is accounted for at the lower resolution.

\section{Experimental Validation of the Theoretical Results}\label{simulations}
We illustrate the theoretical discussions in section~\ref{theoryresults} with problem   \textit{gravity}   from the Regularization toolbox \cite{Regtools}, for which 
\begin{align}
H(s,t) &= \frac{d}{\paren{d^2+(s-t)^2}^{3/2}},\quad (s,t) \in [0,1] \times [0,1], \quad d>0  \label{eq:gravK}\\
f(t) &= \sin\paren{\pi t}+.5\sin\paren{2\pi t}.\label{eq:gravf}
\end{align}
This problem  has the advantage that we can explicitly determine $(\Delta^{(n)})^2$   and  the parameter dependence due to $d$ introduces problems of  ill-posedness increasing with $d$. 
\subsection{Illustration of the SVE-SVD relation}
To investigate the convergence of $(\Delta^{(n)})^2 \rightarrow 0$ we use 
\begin{align*}
\norm{H}_2^2 = \int_0^1\int_0^1  \frac{d^2}{\paren{d^2+(s-t)^2}^3} \,dt\,ds = \frac{3\arctan\paren{\frac{1}{d}}+\frac{d}{d^2+1}}{4d^3}.
\end{align*}
Thus $H$ is square integrable.  For $d=.25$, $\norm{H}_2^2 \approx 67.404$ and for $d=.5$, $\norm{H}_2^2 \approx 7.443$. Moreover, $\sub{a_{ij}}$ in \eqref{kernel matrix} is also available exactly. For $s_{i+1}-s_i=ds=dt =t_{j+1}-t_j$,  
\begin{align*}
\sub{a_{ij}}&=\frac{1}{\sqrt{ds dt}}\int_{s_i}^{s_{i+1}}\int_{t_j}^{t_{j+1}}   \frac{d}{\paren{d^2+(s-t)^2}^{3/2}} \,dt\,ds  \\
&=\frac{1}{dt} \frac{\left( \sqrt{(s_{i+1}-t_j)^2+d^2} + \sqrt{(s_{i}-t_{j+1})^2+d^2} -2 \sqrt{(s_i-t_j)^2+d^2}\right)}{d}.
\end{align*}

In Figure~\ref{fig:delta} we show the convergence of $\lvert \paren{{\Delta^{(n)}}}^2 \rvert$ with $n$ for problem sizes $n=100, \dots, 1000$, with both $d=0.25$ and $d=0.5$. It is worth noting that calculation of this estimate with the midpoint rule generates convergence of the estimate, but due to quadrature error the numerical calculation of $(\sub{\Delta})^2$ using $\|H\|^2 - \|\sub{A}\|_F^2$, is negative.  $\|\sub{A}\|_F^2$ still converges to $\|H\|^2$, but from above rather than from below $\|H\|^2$. For the exact calculation, shown in Figure~\ref{fig:deltaexact}, $\|\sub{A}\|_F^2$ converges  from above as given by the theory. The convergence of the singular values, leading to the effective numerical rank, Theorem~\ref{lemma2}, is illustrated in Figures~\ref{fig:pstara}-\ref{fig:pstarb}. The vertical line indicates the rank calculation for $\epsilon=10^{-15}$. Here, consistent with the simulations $\sub{A}$ is  calculated using the midpoint quadrature rule. The singular values decay exponentially up to the numerical precision of the Matlab implementation, $\epsilon=10^{-16}$. For problem size $n=50$ and $d=0.25$, numerical precision does not impact the singular values. We note that in these, and subsequent figures, the markers and colors are consistently determined by resolution $n$. 

\begin{figure}[H] 
\begin{centering}
\subfigure[Quadrature $\lvert \paren{{\Delta^{(n)}}}^2 \rvert$ \label{fig:delta}]{\includegraphics[width= .24\textwidth]{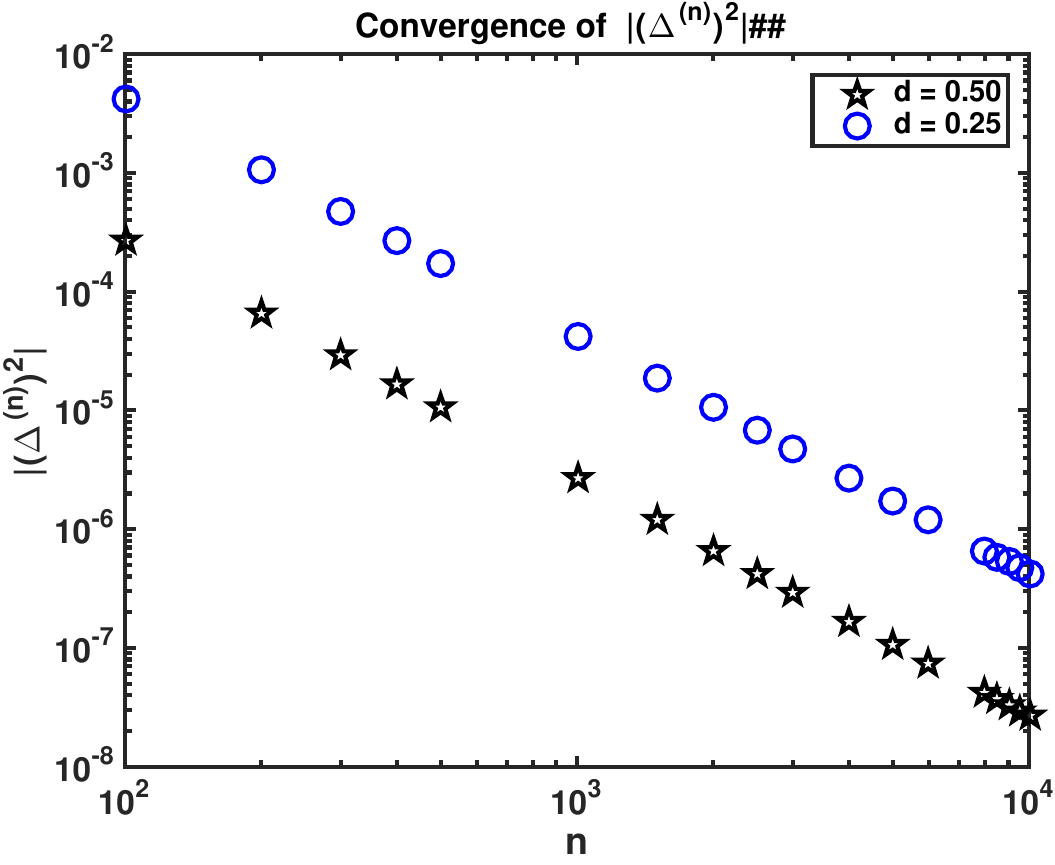}}
\subfigure[Exact $\lvert \paren{{\Delta^{(n)}}}^2 \rvert$ \label{fig:deltaexact}]{\includegraphics[width= .24\textwidth]{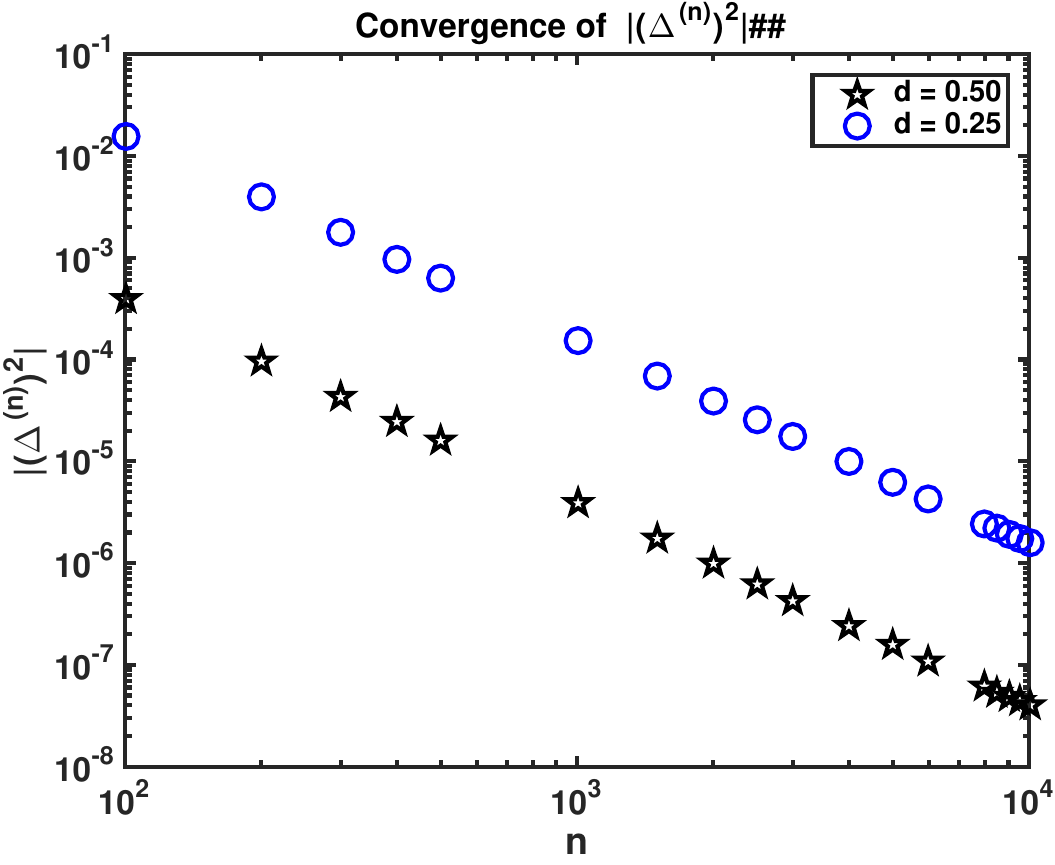}}
\subfigure[$\{\sigma_i^{(n)}\}$ $d=.25$\label{fig:pstara}]{\includegraphics[width= .24\textwidth]{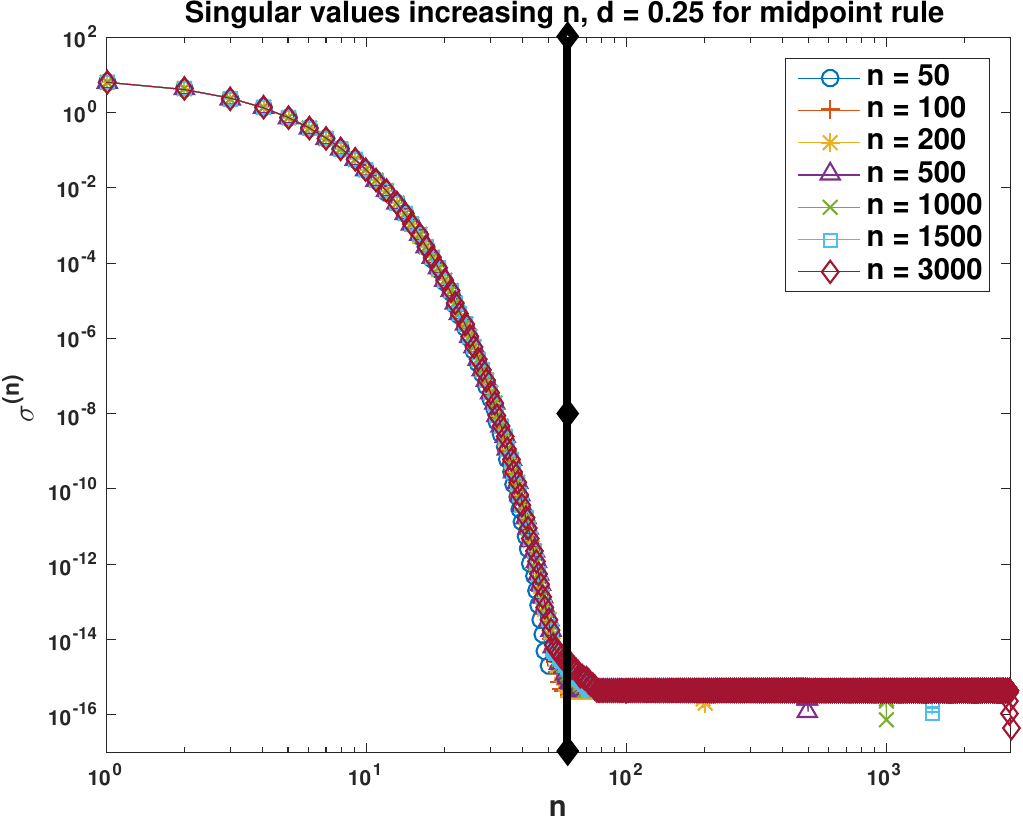}}
\subfigure[$\{\sigma_i^{(n)}\}$ $d=.50$\label{fig:pstarb}]{\includegraphics[width= .24\textwidth]{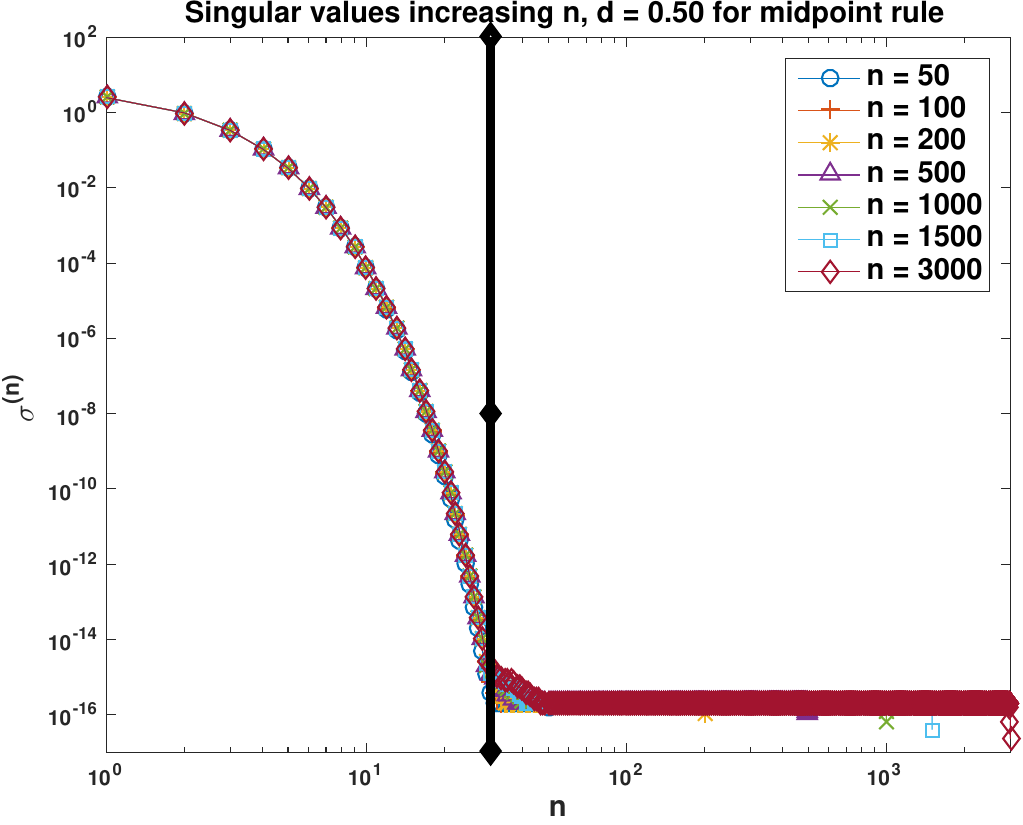}}
\caption{In \ref{fig:delta} $\lvert \paren{{\Delta^{(n)}}}^2 \rvert$ against $n$ for problem \texttt{gravity} with $d=.25$ and $d=.50$, with the matrix approximated using the given quadrature rule. In contrast, \ref{fig:deltaexact} shows the result with the  matrix calculated exactly. In \ref{fig:pstara}-\ref{fig:pstarb} the singular values are plotted  against $n$ with the calculation of $p^*$ for $\epsilon=10^{-15}$ indicated by the vertical line.  Here $n=50$, $100$, $200$, $500$,       $1000$,        $1500$      and   $3000$.  In these, and subsequent figures, the markers and colors are consistently determined by resolution $n$. 
\label{fig:pstar}}
\end{centering}
\end{figure}

\subsection{Illustrating convergence of the functionals}\label{functionals}
We illustrate in Figure~\ref{fig:regfunch1} the convergence of the regularization functionals with truncation at $\sub{p}$ for the \texttt{gravity} problem with $d=0.25$ and $d=0.5$, with discretizations using $n=50$, $100$, $200$, $500$, $1000$, $1500$ and $3000$ points as in Figure~\ref{fig:pstar}. The truncation parameter $\sub{p}$ is determined for $\epsilon=10^{-15}$ in each case. To better identify the location of the root for the MDP and ADP functionals, we plot $\lvert D(\lambda)-p^* \zeta^2\rvert$ and $\lvert C(\lambda)-p^* \zeta^2\rvert$, i.e. assuming safety parameter $\tau=1$ for the MDP.  In the calculations the coefficients $\sub{\beta_i}$ and $\sub{\sigma_i}$ are weighted by the noise in the measurements of $g(s)$, i.e. corresponding to whitening the system by the inverse square root of the covariance matrix of the measured noise in the data.  We pick a constant value of $\zeta^2$ to verify the convergence of the functionals for constant $\zeta^2$ as discussed in section~\ref{theoryresults}. Here, for the MDP, ADP and UPRE $\zeta^2=ds^{(50)}$, i.e. the noise at the coarsest resolution of the problem.  

The UPRE and GCV functionals converge  independent of $n$, while the MDP and ADP are clearly more sensitive to the correct choice of $\zeta^2$ with $n$. On the other hand, the UPRE and GCV functionals become increasingly flat with decreasing $\lambda$ suggesting that the determination of the minimum will be difficult. One can see that the MDP suggests a larger regularization parameter, with this choice of $\tau$, and will likely oversmooth as compared to UPRE and GCV, while the ADP parameter may lead to undersmoothing. 

\begin{figure} 
\begin{centering}
\subfigure[$d=.25$]{\includegraphics[width= .45\textwidth]{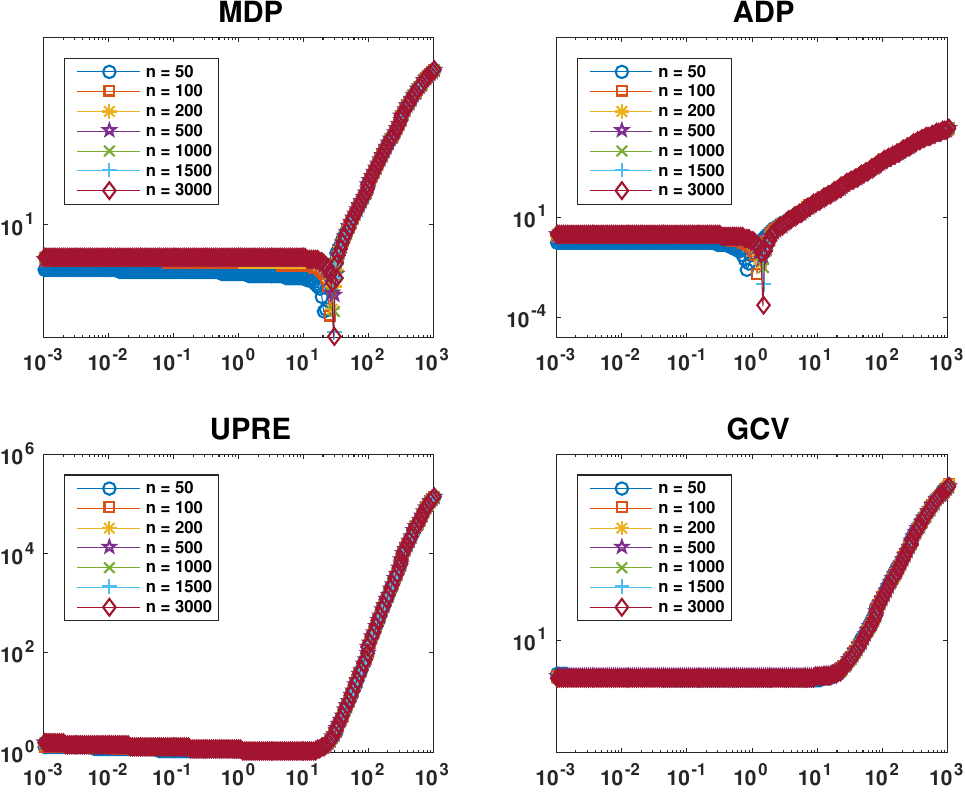}}
\subfigure[$d=.25$]{\includegraphics[width= .45\textwidth]{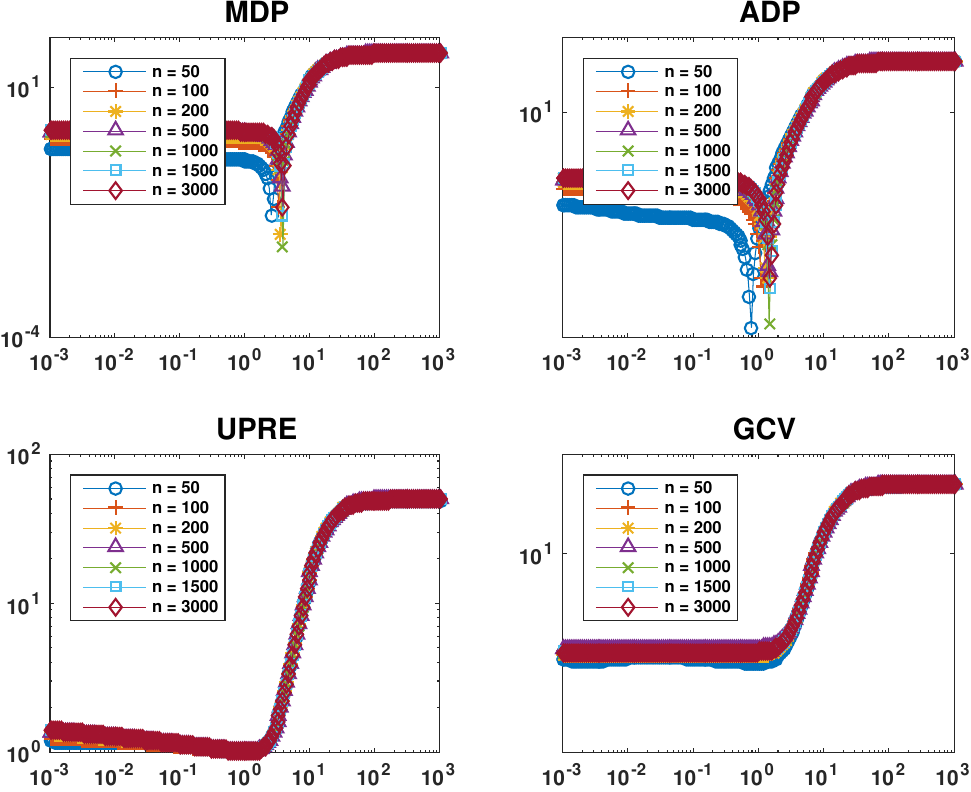}}\\
\subfigure[$d=.50$]{\includegraphics[width= .45\textwidth]{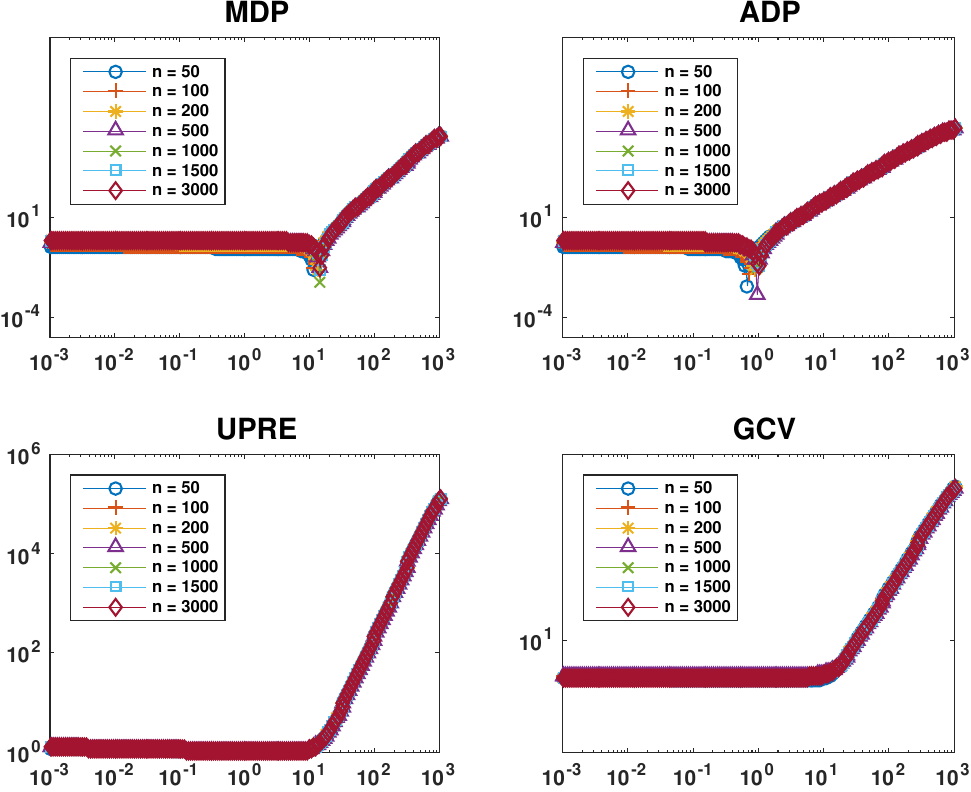}}
\subfigure[$d=.50$]{\includegraphics[width= .45\textwidth]{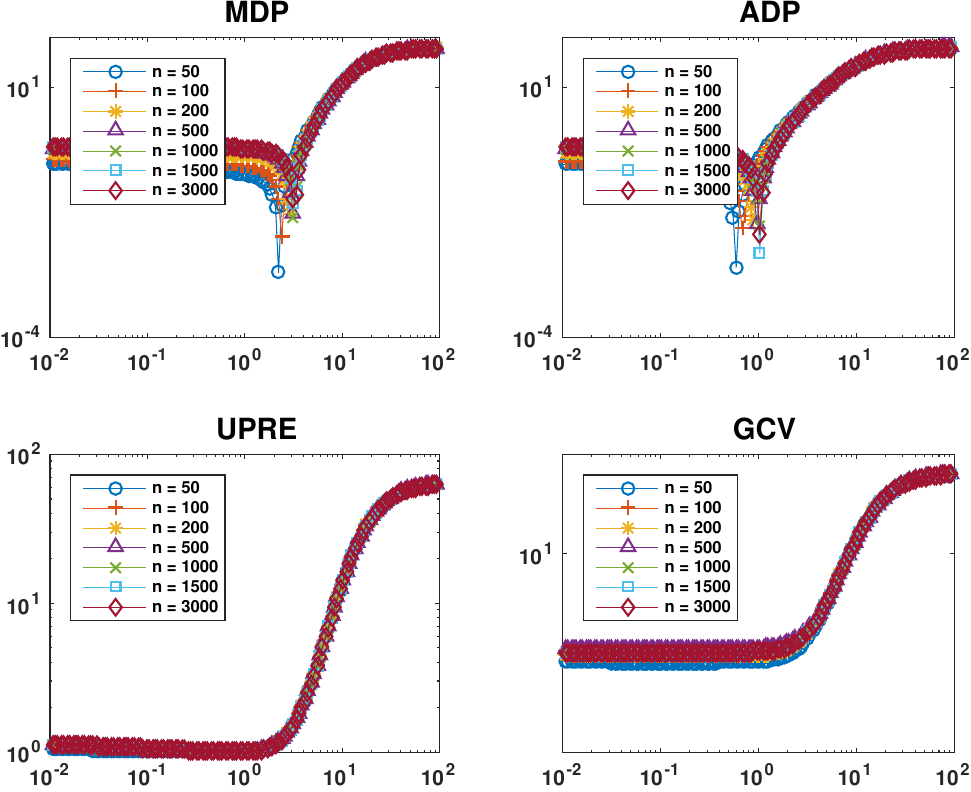}}
\caption{The MDP, ADP, UPRE, and GCV,   functionals  \eqref{TMDP}-\eqref{Tgcv} for data weighted by the inverse covariance matrix of the noise in the data, with  $p^*$ determined for $\epsilon=10^{-15}$, for $d=0.25$, and $d=0.5$. Here in all cases $\zeta^2=ds^{(50)}$.  Note that for the MDP and ADP we plot $\lvert D(\lambda)-p^*\zeta^2\rvert$ and $\lvert C(\lambda)-p^* \zeta^2\rvert$.  \label{fig:regfunch1}}
\end{centering}
\end{figure}

Figure~\ref{fig:regfunchi} illustrates the regularization functionals, calculated as for Figure~\ref{fig:regfunch1} but with $\zeta^2=ds^{(n)}$ for each $n$, ie at the correct variance in each case.  Now one more clearly sees the movement of the MDP and ADP curves to the left, corresponding to $\sub{\lambda} \rightarrow 0$ with $n$, indicative of $\zeta^2=ds^{(n)}$ also converging to $0$ with $n$. On the other hand, the plots for the UPRE and the GCV show that the residual term dominates in each case, so that the curves exhibit the same convergence shown in Figure~\ref{fig:regfunch1}, further emphasizing the potential difficult of minimizing the functionals at low noise levels. Still, the UPRE in Figure~\ref{fig:regfunch1} indicates a greater evidence of a minimum in the given range. It should be noted that the plots for the GCV are independent of the choice of the variance. 

The vertical lines in each case indicate the location of the regularization parameter with $n$, calculated using the scaling argument in \eqref{scaling}, but imposing a fixed variance, e.g. $\left(\zeta^{(50)}\right)^2$ for all $n$ to find $\sub{\tilde{\lambda}}$. 
 In this case $\subl{\lambda}$ decreases with $n$, since it is calculated for decreasing variance $\left(\zeta^{(n)}\right)^2$ with $n$, as $\sub{d s} \rightarrow 0$. 
For the GCV the vertical lines indicate the regularization parameters chosen for each $n$, rather than from $n=50$ scaled to larger $n$. The locations of $\sub{\lambda}$ for the MDP give quite good estimates for the locations of the minima of the relevant functions, but the ADP estimates tend to be larger than would be suggested by the minima, and hence that there may be less under-smoothing when calculated from the case with $n=50$. The difficulty with estimating a good minimum for the GCV is evident. One should expect that  $\sub{\lambda}$ tends to the left with increasing $n$, but this characteristic is not always observed, indeed no monotonicity in $\sub{\lambda}$ is found. 
With the GCV the terms using $\sub{\beta_i}$ for $i>p^*$  converge to $0$, because these coefficients represent the less dominant spectral coefficients in the expansion for $g(s)$, the dominant energy is maintained in the first $p^*$ terms, independent of $n$. Additionally, this means that with $\zeta^2\rightarrow 0$, the UPRE and GCV functionals are both minimizing the residual only scaled by a different constant term, and finding $\sub{\lambda}$ directly with $\zeta^2=ds^{(n)}$ effectively minimizes the residual and may lead to under-smoothing in the solution. We note that  determining the correct tolerance for finding the minimum of the ADP and MDP, formulated as minimization of the distance from the right hand side,  is a limiting factor of  both ADP and MDP methods.

\begin{figure} 
\begin{centering}
\subfigure[$d=.25$]{\includegraphics[width= .45\textwidth]{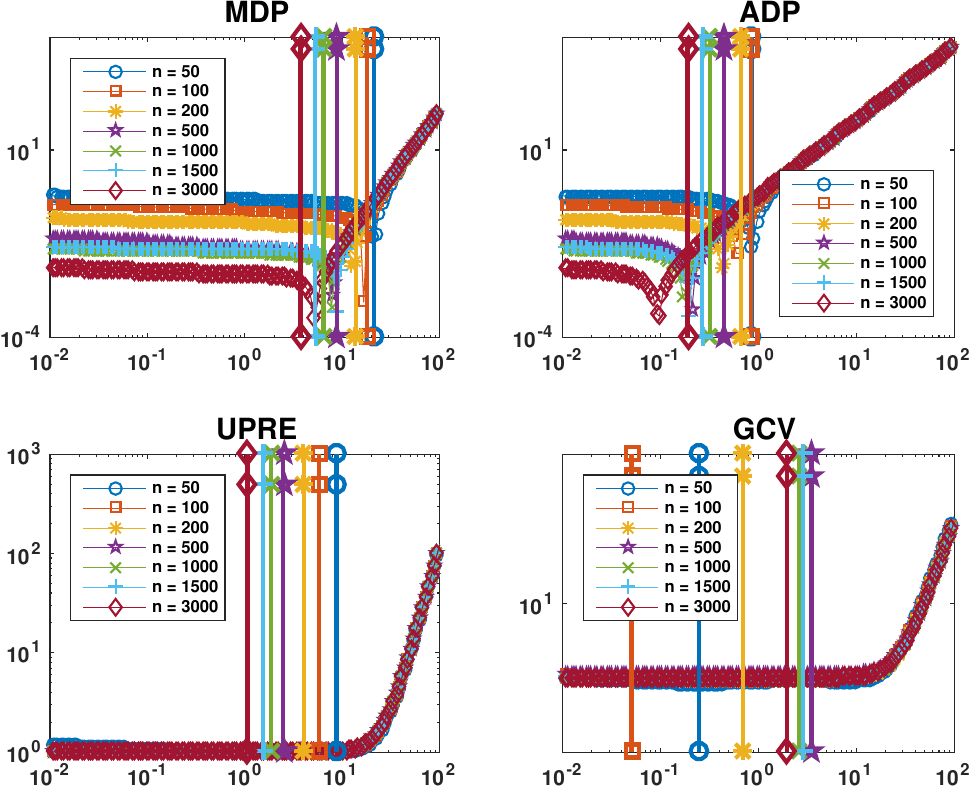}}
\subfigure[$d=.25$]{\includegraphics[width= .45\textwidth]{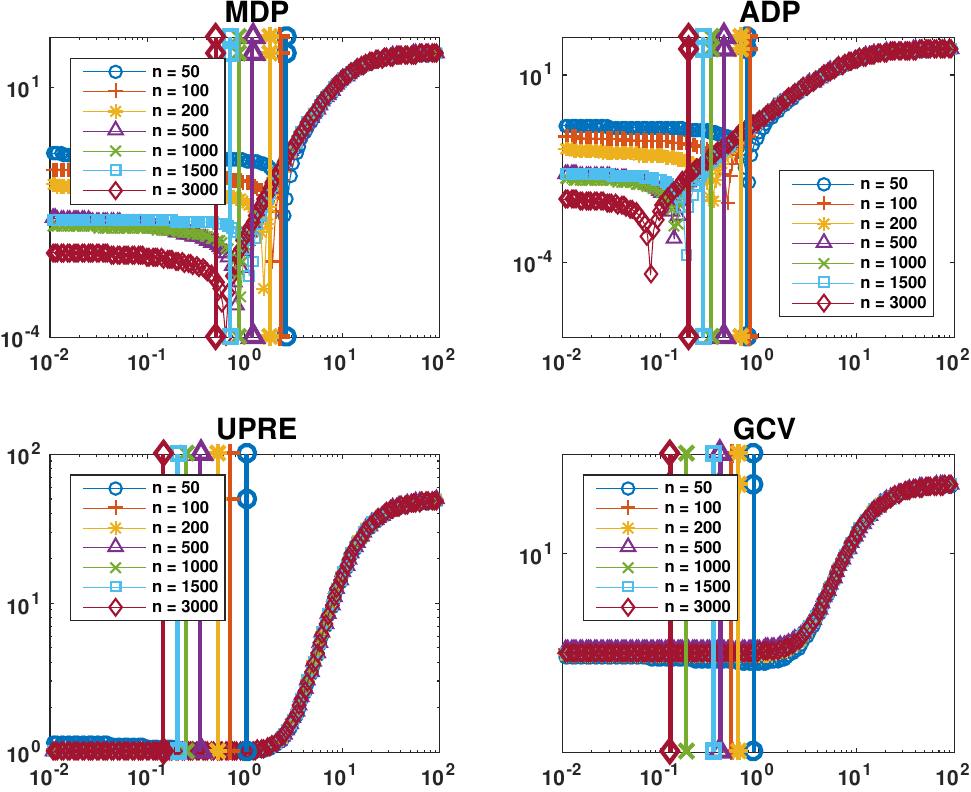}}\\
\subfigure[$d=.50$]{\includegraphics[width= .45\textwidth]{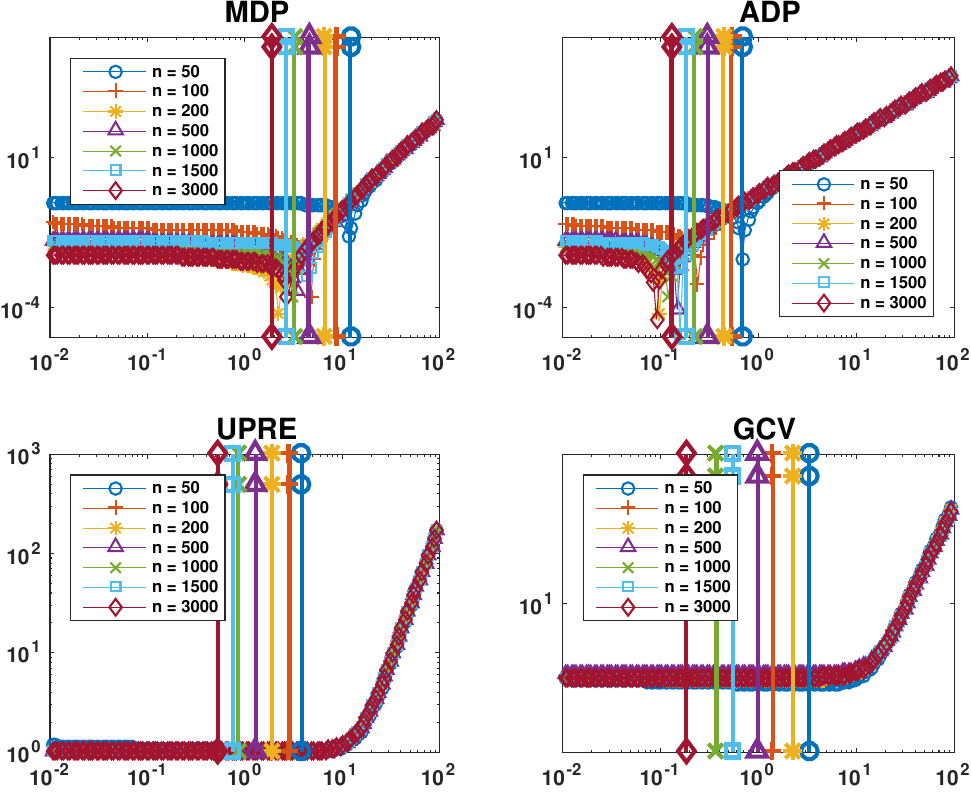}}
\subfigure[$d=.50$]{\includegraphics[width= .45\textwidth]{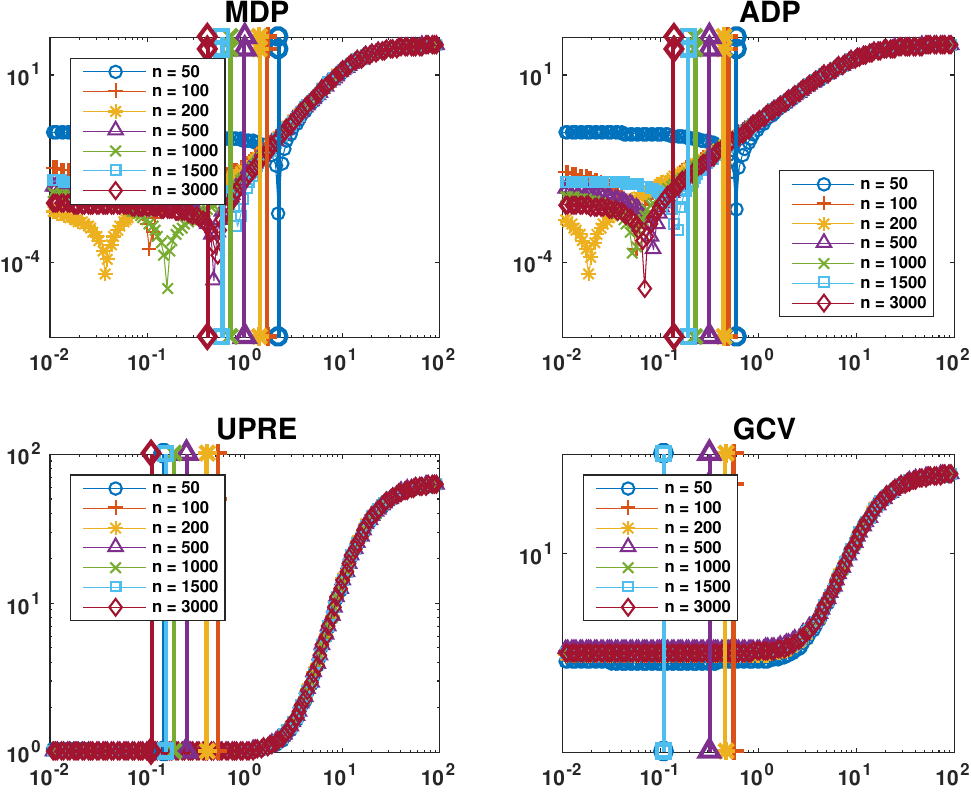}}
\caption{The MDP, ADP, UPRE, and GCV,   functionals  \eqref{TMDP}-\eqref{Tgcv} for data weighted by the inverse covariance matrix of the noise in the data, with  $p^*$ determined for $\epsilon=10^{-15}$, for $d=0.25$, and $d=0.5$. Here in all cases $\zeta^2=ds^{(n)}$, dependent on $n$.  Note that for the MDP and ADP we plot $\lvert D(\lambda)-p^*\zeta^2\rvert$ and $\lvert C(\lambda)-p^* \zeta^2\rvert$. The vertical lines in each case indicate the location of the regularization parameter with $n$, calculated using the scaling argument in \eqref{scaling} using variance $\left(\zeta^{(50)}\right)^2$ for all $n$ to find $\sub{\tilde{\lambda}}$.  \label{fig:regfunchi}}
\end{centering}
\end{figure}

\section{Numerical Experiments}\label{results}
We present a selection of results using Algorithm~\ref{alg} to demonstrate its use for problems with differing characteristics. First we look further at problem \texttt{gravity} with two different levels of conditioning as determined by $d=0.25$ and $d=0.50$. We  also show the results  using \texttt{gravity} for a  discontinuous source. Note that for \texttt{gravity} the spectrum decays very quickly. In contrast,  problem \texttt{deriv2}, also from \cite{Regtools}, has a very slowly decaying spectrum. Finally, therefore, we illustrate simulations for high noise and \texttt{deriv2}. The presented results are illustrative of our experiments with other samples, noise levels and problems and are  given to verify the approach. 
\subsection{Problem \texttt{gravity}}\label{results_1}
We consider a problem of size $N=3000$ for the discretization of  \eqref{eq:gravK}-\eqref{eq:gravf}. Regularization parameters to use for $N=3000$ are obtained by downsampling the data, as described in Algorithm~\ref{alg}, to problems of size $n=50$, $100$, $200$, $500$, $1000$ and $1500$, representing sampling the data at intervals $\ell=60$, $30$, $15$, $6$, $3$ and $2$.  For comparison the solutions without any downsampling, i.e. using $N=3000$ and $\ell=1$ are also provided. Noisy data are obtained by forming $g_\mathrm{obs}(s_i) = g(s_i)+\nu \max_j(|g(s_j)|) e(s_i)$,  for $\nu=0.001$ and $\nu = 0.1$, representing low and high noise. For $d=0.25$ and $d=0.5$ respectively, $\max_j |g(s_j)| = 6.7542$ and $2.1895$.  Errors $e(s_i)$ are drawn from a random normal distribution with variance $1$. The resulting data samples  are illustrated in Figure~\ref{fig:noisy}. 
\begin{figure}
\begin{center}
\subfigure[$d=.25$]{\includegraphics[width=.24\textwidth]{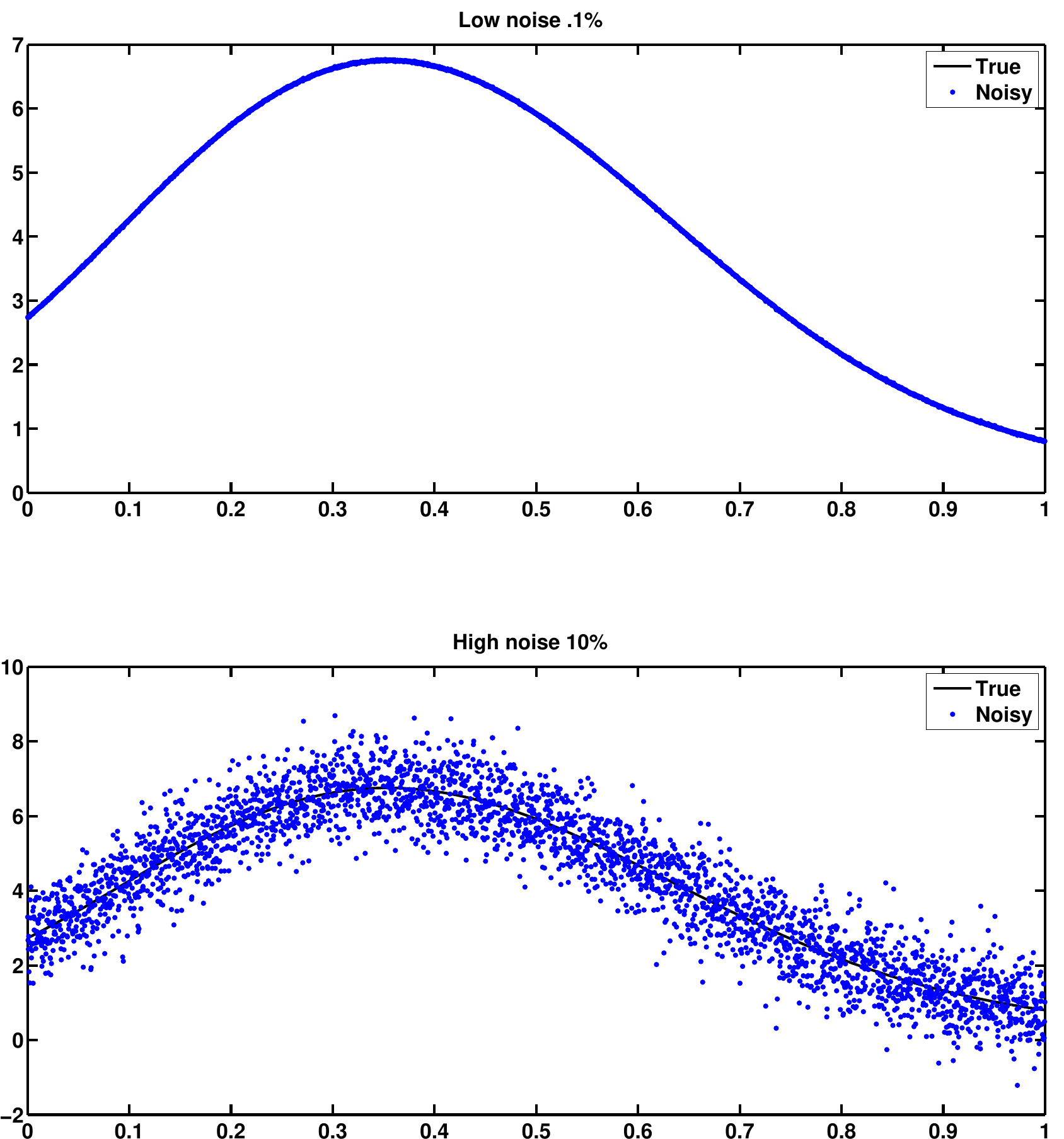}}
\subfigure[$d=.50$]{\includegraphics[width=.24\textwidth]{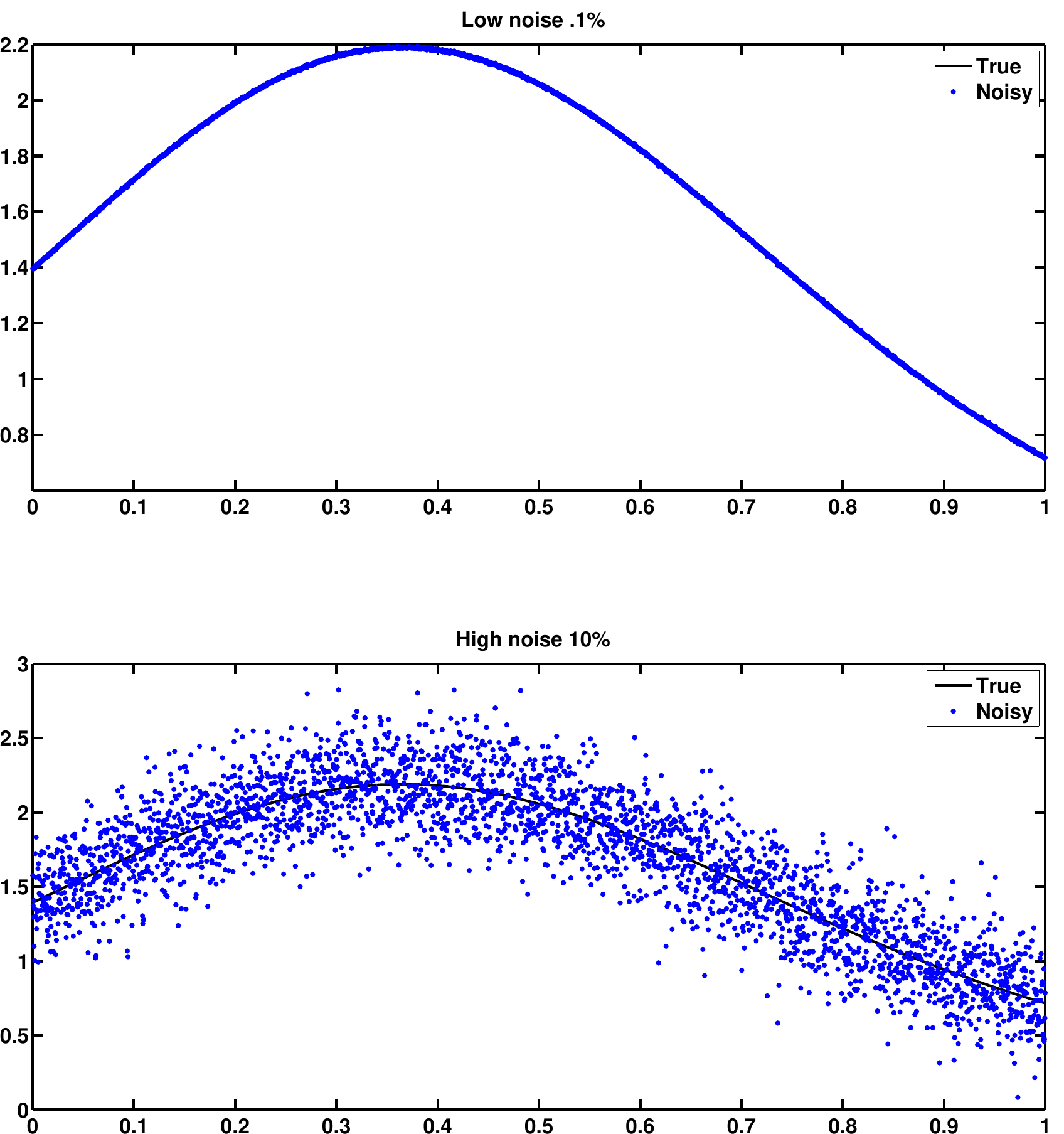}}
\end{center}
\caption{Illustrative noisy data with $.1\%$ and $10\%$ noise for problem \texttt{gravity} with $N=3000$.  \label{fig:noisy}}
\end{figure}
%
 
We note that  the estimate for  $p^*$ which is approximated by $\sub{p}$ may not be stable initially for small $n$, thus potentially leading to different estimates of the regularization parameter to use for $\subl{\lambda}$, when estimated using different values for $n$, ie different samplings. Although one  may theoretically chose to determine $p^*$  for any given $\epsilon$, here we present results corresponding to $\epsilon=10^{-15}$, which as can be seen from Figure~\ref{fig:pstar} is effectively the point at which the  singular values for $i>p^*$ are contaminated by numerical noise. In all cases the regularization parameter at resolution $n$ is calculated by each of the regularization parameter estimation methods, MDP, ADP, UPRE and GCV using $\zeta^2=ds^{(n)}$. The estimate with the GCV is independent of the given $\zeta^2$. Given $\sub{\lambda}$ the fine solution using $3000$ points is calculated using \eqref{scaling} and the dominant $p^*=\sub{p}$ components of the SVD for the matrix $\subl{A}$. 

We first illustrate in Figure~\ref{fig:soln}  the solutions for a single arbitrarily chosen noise vector. All solutions are calculated using $3000$ points, but for clarity in the plots the solutions are plotted using just $50$ points. Any solution which has an amplitude greater than $3$  is also omitted from the plots, in order that the plots are not cluttered by the high oscillatory behavior of the severely unstable solutions. Solutions with less severe instability are also evident as moderate oscillations around the true solutions.  Note that the actual amplitude of the exact solution is less than $1.5$. In each case the individual legends indicate which solutions are plotted. 
\begin{figure}
\begin{centering}
\subfigure[$d=.25$, $.1\%$ noise ]{\includegraphics[width= .49\textwidth]{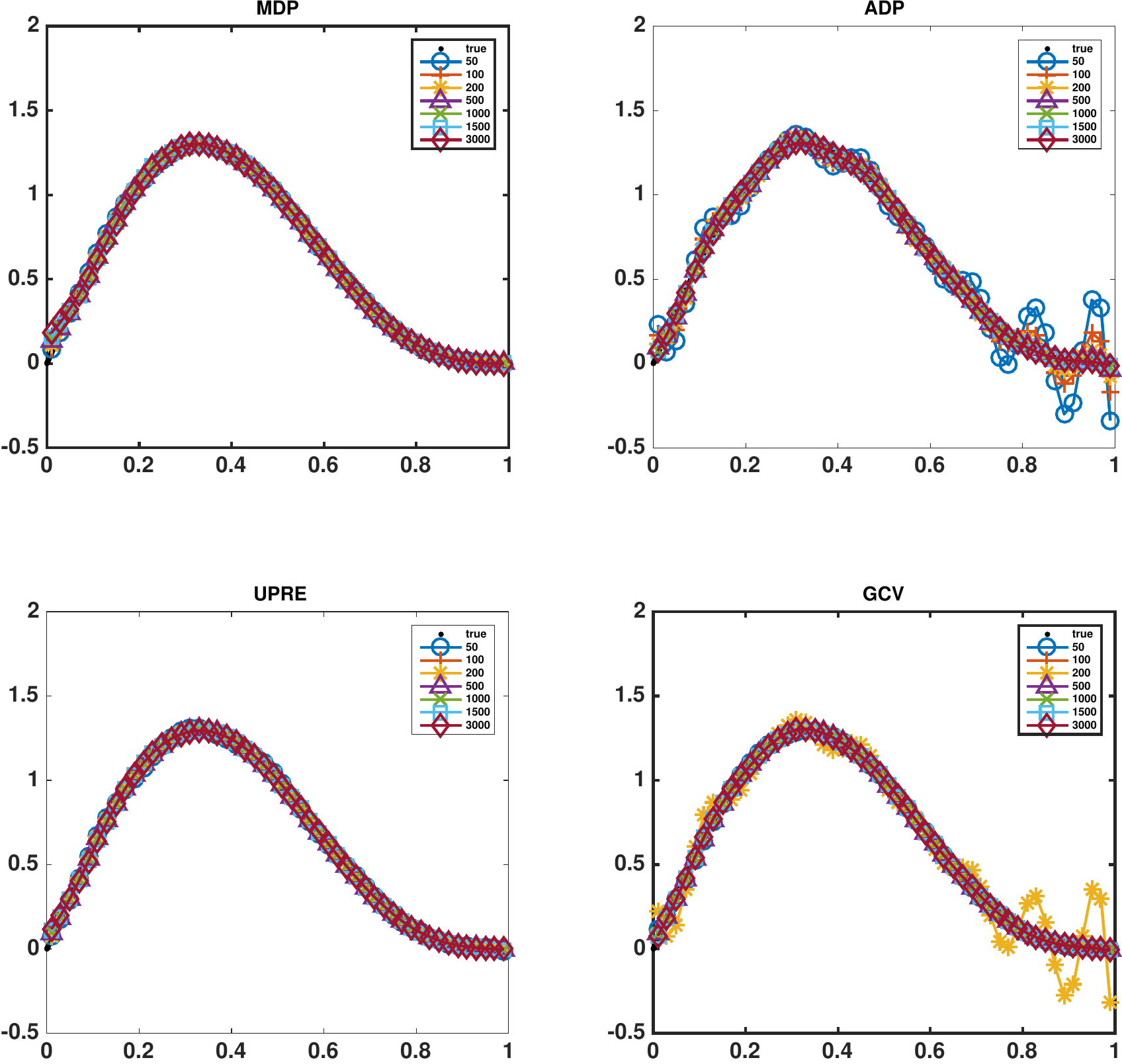}} 
\subfigure[$d=.25$, $10\%$ noise ]{\includegraphics[width= .49\textwidth]{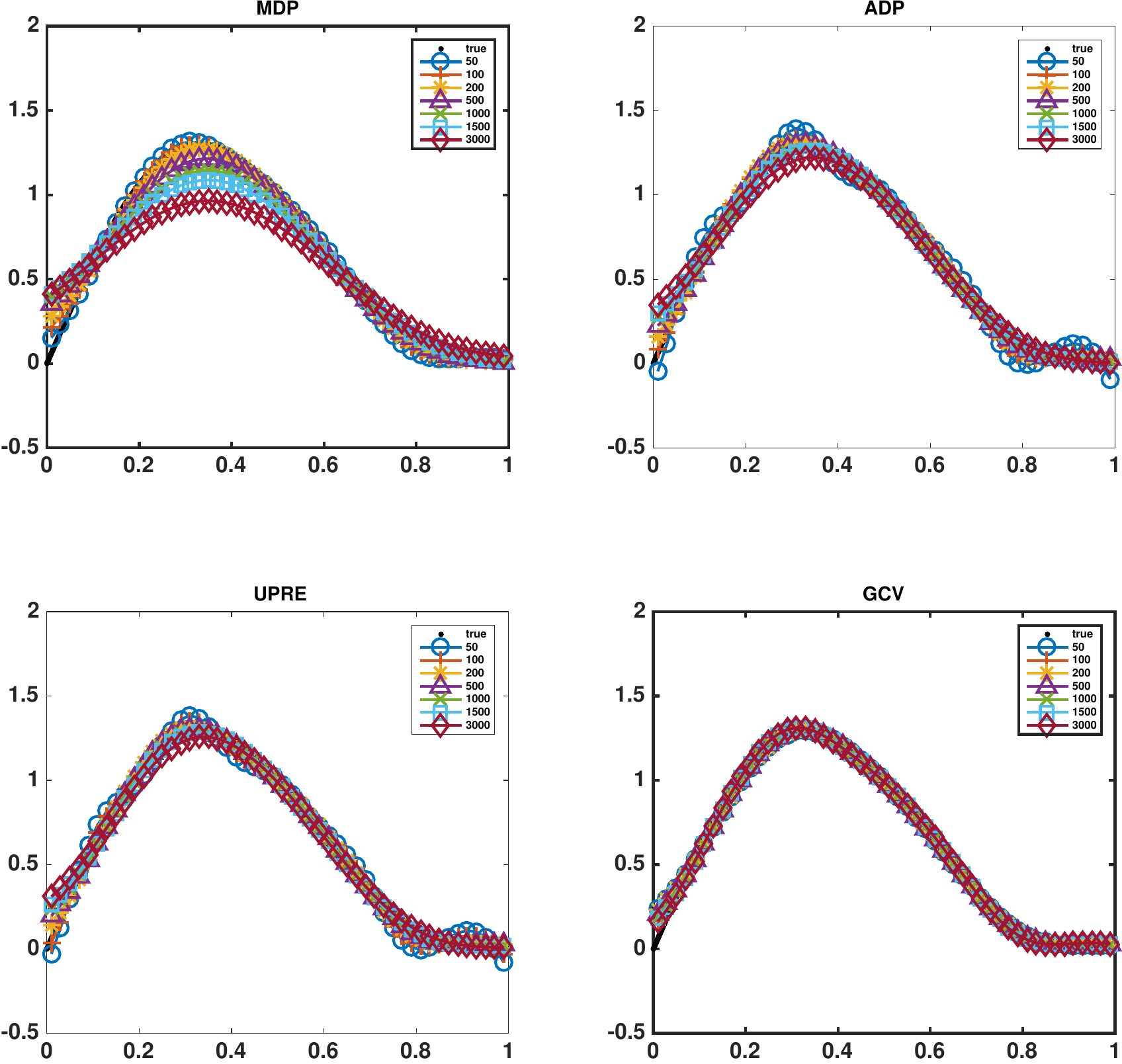}} 
\subfigure[$d=.50$, $ .1\%$ noise]{\includegraphics[width= .49\textwidth]{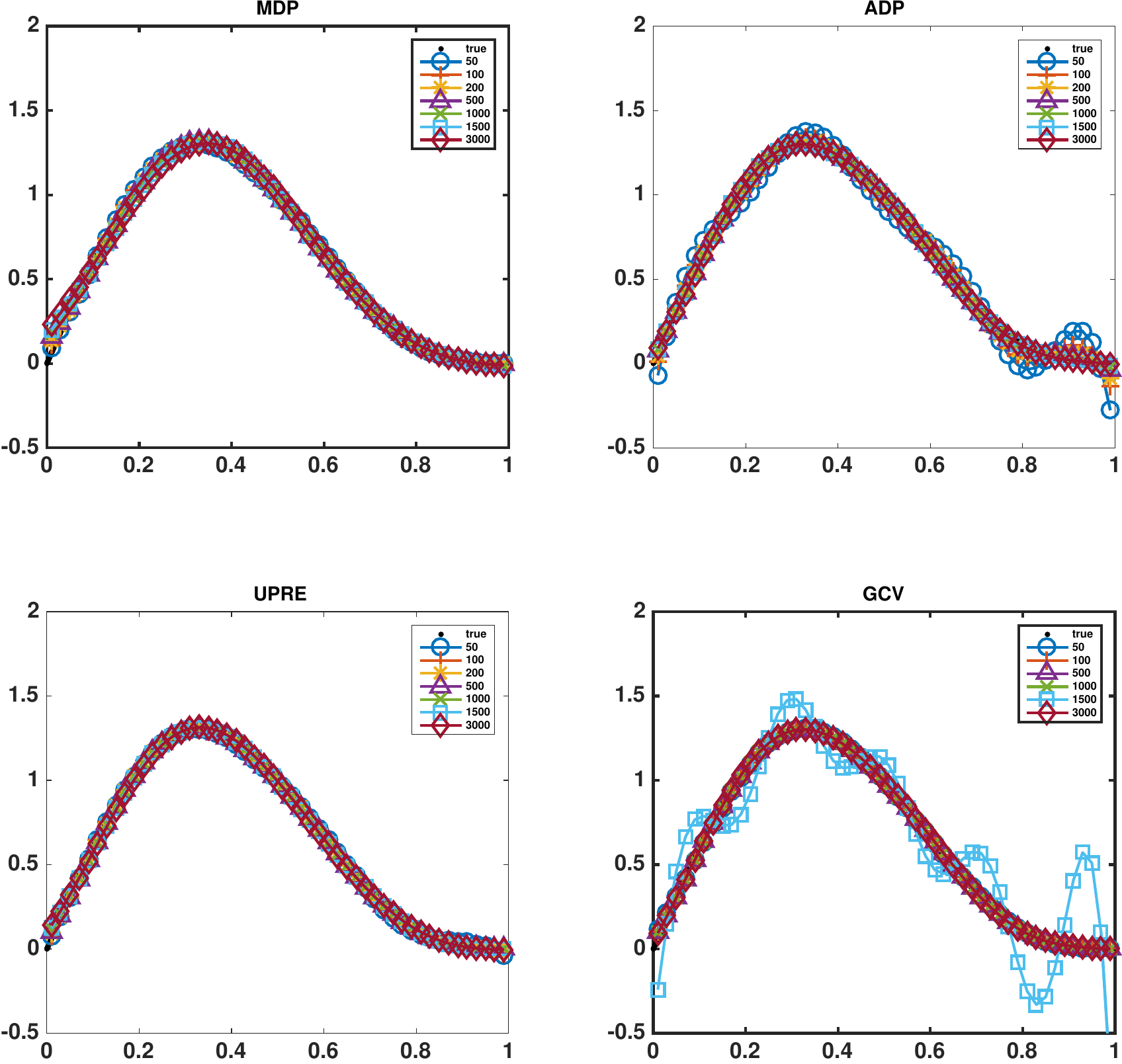}} 
\subfigure[$d=.50$, $10\%$ noise]{\includegraphics[width= .49\textwidth]{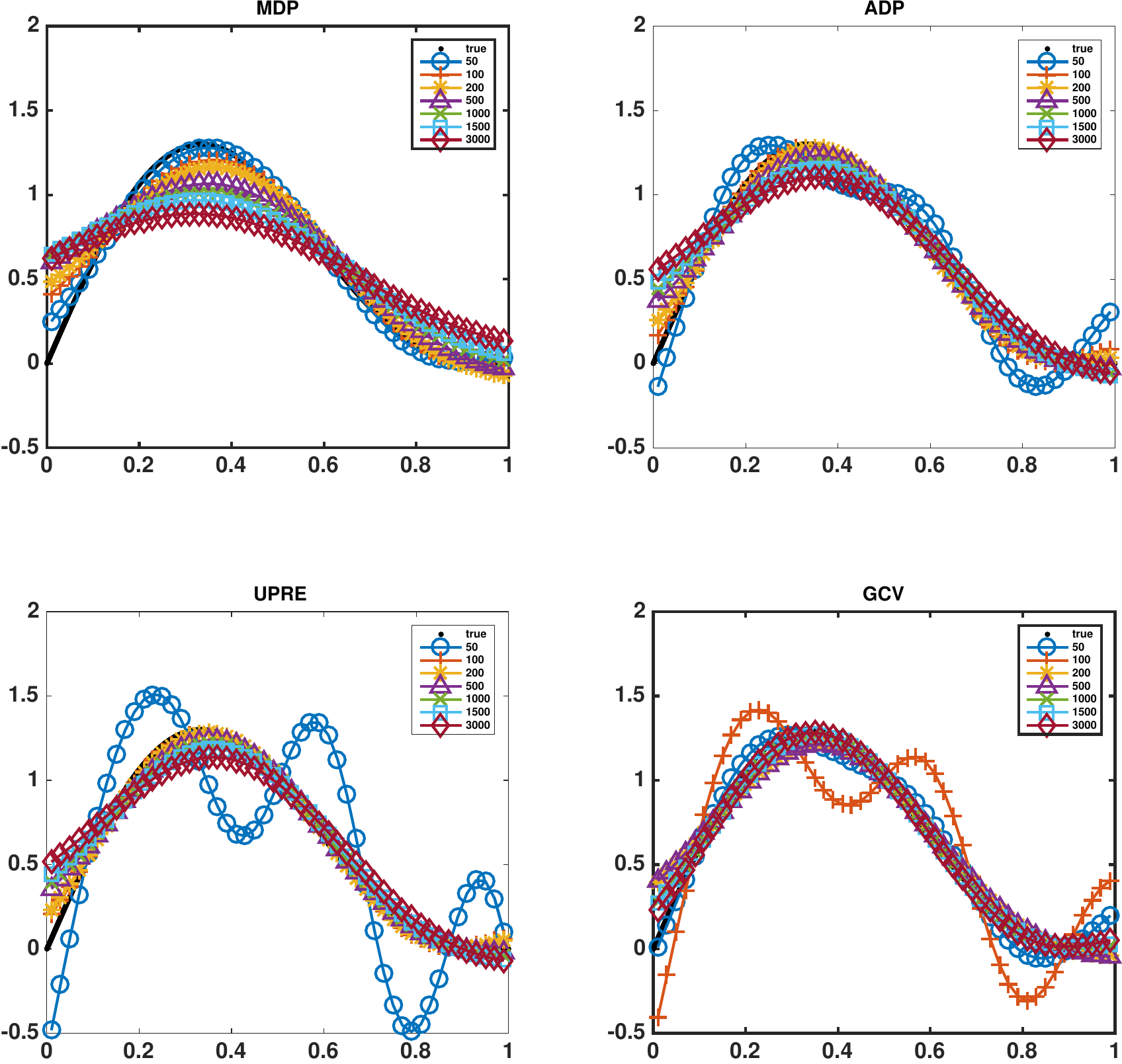}} 
\caption{Solutions of problem \texttt{gravity} for the data illustrated in Figure~\ref{fig:noisy} using Algorithm~\ref{alg}. Here in order to improve clarity, only those solutions which have a maximum amplitude less than $3$ are shown.   The legend in each case indicates which solutions satisfy this constraint.  In all cases the regularization parameter is found at the given $n$ and used to provide the regularization parameter for $N=3000$, e.g. we find the regularization parameter with $n=50$, $100$, $200$, $500$, $1000$, $1500$ and $3000$ and then produce the solution at $3000$ points with this regularization parameter. The solution found using $N=3000$ is illustrated to demonstrate the quality of the solutions obtained for $n<3000$. To compare solutions, they are all plotted at indices $30$ to $2970$ at increments of $60$ for the solution of length $3000$.  
\label{fig:soln}}
\end{centering}
\end{figure}
The most immediate observation is that no one single method is perfect for all $n$, but this is not at all surprising; methods for estimating regularization parameters  are not foolproof and it would be deceptive to indicate otherwise. The success of any given method depends on the specific given  right hand side data. An example of this is shown in the case with $.10\%$ noise and $d=0.50$ for which  GCV   generates  acceptable solutions for all $n$, except $n=200$, which is not shown. Further the solution obtained by the ADP for the cases with low noise are less stable for $N=3000$, which  illustrates that the estimate for $\lambda^{(N)}$ obtained from $n<N$ may yield more stable solutions. Overall,  these results confirm that  GCV is less reliable, as already suggested from the discussion of the functionals in section~\ref{functionals}. On the other hand, in most cases, except for high noise and $d=0.5$, the regularization parameter estimated at the coarsest level, with $n=50$ giving $\lambda^{(50)}$  and then used to obtain the solution for size $N=3000$,  is remarkably robust.

 \begin{table}
\caption{Relative errors, mean(standard deviation) over $25$ samples for $d=0.25$ and $0.10\%$ noise.  Solutions calculated for resolution with $3000$ points using the regularization parameter calculated using $n=50$, $100$, $200$, $500$, $1000$, $1500$ and $3000$ points, using $\zeta^2=ds^{(n)}$ in the estimation of $\sub{\lambda}$. The minimum average relative error by each method for $n<3000$ in bold face. \label{Errord100noise1}}
\begin{tabular}{|*{
9}{c|}} \hline\noalign{\smallskip}
$n$& ADP & MDP &UPRE &GCV  \\ \noalign{\smallskip}\hline\noalign{\smallskip}
$50 $&$7.4813(37.523)$&$7.2614(37.566)$&$0.0752(0.207)$&$1.0691(5.442)$\\ 
$100 $&$0.0992(0.033)$&$\mathbf{0.0104(0.002)}$&$0.0156(0.006)$&$0.0369(0.091)$\\ 
$200 $&$0.0606(0.020)$&$0.0129(0.001)$&$0.0109(0.003)$&$0.2521(1.048)$\\ 
$500 $&$0.0354(0.010)$&$0.0169(0.001)$&$\mathbf{0.0097(0.002)}$&$0.0166(0.017)$\\ 
$1000 $&$0.0242(0.006)$&$0.0208(0.001)$&$0.0108(0.002)$&$\mathbf{0.0147(0.018)}$\\ 
$1500 $&$\mathbf{0.0196(0.005)}$&$0.0234(0.000)$&$0.0119(0.002)$&$0.0310(0.061)$\\ 
$3000 $&$0.0141(0.004)$&$0.0289(0.000)$&$0.0142(0.001)$&$0.0364(0.074)$\\ 
\hline\noalign{\smallskip}
\end{tabular}\end{table}
\begin{table}
\caption{Relative errors, mean(standard deviation) over $25$ samples for $d=0.25$ and $10.00\%$ noise.  Solutions calculated for resolution with $3000$ points using the regularization parameter calculated using $n=50$, $100$, $200$, $500$, $1000$, $1500$ and $3000$ points, using $\zeta^2=ds^{(n)}$. The minimum average relative error by each method for $n<3000$ in bold face. \label{Errord100noise100}}
\begin{tabular}{|*{
9}{c|}} \hline\noalign{\smallskip}
$n$& ADP & MDP &UPRE &GCV  \\ \noalign{\smallskip}\hline\noalign{\smallskip}
$50 $&$8.5622(28.062)$&$7.0155(25.881)$&$4.7672(20.958)$&$0.1437(0.281)$\\ 
$100 $&$0.0981(0.038)$&$\mathbf{0.0512(0.010)}$&$0.1040(0.039)$&$0.1615(0.369)$\\ 
$200 $&$0.0608(0.019)$&$0.0642(0.007)$&$0.0714(0.026)$&$0.3365(1.004)$\\ 
$500 $&$\mathbf{0.0511(0.010)}$&$0.0978(0.004)$&$\mathbf{0.0522(0.014)}$&$0.3274(1.385)$\\ 
$1000 $&$0.0606(0.008)$&$0.1376(0.003)$&$0.0525(0.010)$&$\mathbf{0.1082(0.169)}$\\ 
$1500 $&$0.0705(0.006)$&$0.1689(0.003)$&$0.0579(0.008)$&$0.3233(0.838)$\\ 
$3000 $&$0.0949(0.004)$&$0.2429(0.003)$&$0.0742(0.006)$&$0.0939(0.073)$\\ 
\hline\noalign{\smallskip}
\end{tabular}\end{table}
\begin{table}
\caption{Relative errors, mean(standard deviation) over $25$ samples for $d=0.50$ and $0.10\%$ noise.  Solutions calculated for resolution with $3000$ points using the regularization parameter calculated using $n=50$, $100$, $200$, $500$, $1000$, $1500$ and $3000$ points, using $\zeta^2=ds^{(n)}$. The minimum average relative error by each method for $n<3000$ in bold face. \label{Errord200noise1}}
\begin{tabular}{|*{
9}{c|}} \hline\noalign{\smallskip}
$n$& ADP & MDP &UPRE &GCV  \\ \noalign{\smallskip}\hline\noalign{\smallskip}
$50 $&$0.2758(0.211)$&$0.0376(0.071)$&$1.5614(6.356)$&$\mathbf{0.0790(0.221)}$\\ 
$100 $&$0.1046(0.051)$&$\mathbf{0.0147(0.004)}$&$0.0234(0.016)$&$0.1837(0.593)$\\ 
$200 $&$0.0628(0.030)$&$0.0199(0.002)$&$0.0145(0.006)$&$1.3135(3.823)$\\ 
$500 $&$0.0351(0.017)$&$0.0282(0.001)$&$\mathbf{0.0131(0.004)}$&$0.9100(3.662)$\\ 
$1000 $&$0.0237(0.011)$&$0.0352(0.001)$&$0.0154(0.003)$&$0.2765(0.844)$\\ 
$1500 $&$\mathbf{0.0194(0.009)}$&$0.0395(0.001)$&$0.0176(0.003)$&$1.0593(3.822)$\\ 
$3000 $&$0.0148(0.006)$&$0.0487(0.001)$&$0.0226(0.002)$&$1.1728(5.080)$\\ 
\hline\noalign{\smallskip}
\end{tabular}\end{table}
\begin{table}
\caption{Relative errors, mean(standard deviation) over $25$ samples for $d=0.50$ and $10.00\%$ noise.  Solutions calculated for resolution with $3000$ points using the regularization parameter calculated using $n=50$, $100$, $200$, $500$, $1000$, $1500$ and $3000$ points, using $\zeta^2=ds^{(n)}$ .The minimum average relative error by each method for $n<3000$ in bold face. \label{Errord200noise100}}
\begin{tabular}{|*{
9}{c|}} \hline\noalign{\smallskip}
$n$& ADP & MDP &UPRE &GCV  \\ \noalign{\smallskip}\hline\noalign{\smallskip}
$50 $&$4.4299(12.051)$&$2.7538(7.493)$&$3.4666(9.399)$&$0.8003(2.611)$\\ 
$100 $&$0.0974(0.051)$&$\mathbf{0.1212(0.017)}$&$0.1056(0.055)$&$0.4049(1.295)$\\ 
$200 $&$\mathbf{0.0843(0.032)}$&$0.1623(0.011)$&$\mathbf{0.0845(0.039)}$&$\mathbf{0.2221(0.500)}$\\ 
$500 $&$0.1077(0.020)$&$0.2142(0.006)$&$0.0965(0.024)$&$0.6397(1.657)$\\ 
$1000 $&$0.1397(0.014)$&$0.2555(0.005)$&$0.1222(0.017)$&$0.5330(1.379)$\\ 
$1500 $&$0.1591(0.011)$&$0.2809(0.004)$&$0.1399(0.014)$&$2.0134(6.776)$\\ 
$3000 $&$0.1932(0.007)$&$0.3295(0.003)$&$0.1723(0.011)$&$2.5886(7.040)$\\ 
\hline\noalign{\smallskip}
\end{tabular}\end{table}

To further examine the performance of regularization parameter estimation based on the approximate singular expansion, the experiment illustrated in Figure~\ref{fig:noisy} was repeated for $25$ distinct representations of the noise vectors. The approach for calculating the regularization parameter and solution  uses Algorithm~\ref{alg} for each noise vector. The relative error was calculated with respect to the known true solution, and both mean and standard deviation for errors over all $25$ realizations of the noise were obtained, and are reported in Tables~\ref{Errord100noise1}, \ref{Errord100noise100}, \ref{Errord200noise1}, and \ref{Errord200noise100}, for $d=0.25$ and noise $.1\%$ and $10\%$, and then $d=0.5$ and the two noise levels, respectively.  The results validate the expectation from examination of the functionals. Overall the GCV is generally less robust. In all cases the approaches can be more robust for estimating $\subl{\lambda}$ when obtained from subsampling.  Examination of the obtained regularization parameters, not shown here, also confirms that the GCV obtained results show greater variability, even though it should be noted that we use an expensive approach with sampling across $1000$ choices of $\lambda$ in the given range and then seek the minimum around the minimal value found using Matlab function \texttt{fminbnd}.

\subsection{Piecewise constant solution}
It is well-known that regularization parameter estimation is more challenging for non smooth solutions, namely for those for which the exact spectral coefficients  do not decay  quickly to $0$.  The contamination of  spectral coefficients associated with higher frequencies in the basis, i.e. the vectors $\bo v_i$ for larger $i$,  limits the ability to accurately resolve discontinuities in the solutions, and alternative regularizing norms are required, e.g. total variation, iterative regularization etc, e.g. \cite{SB,vogel:02,WoRo:07,Zhd}.  At the heart of such techniques, however, are standard  Tikhonov regularizers.  Thus the ability to obtain reasonable estimates of the solution, even in the presence of sharp gradients and/or discontinuities is still relevant as the component of a more general regularization approach. Indeed, effective regularizers that can be obtained efficiently are even more significant within the context of an iteratively refined Tikhonov solution for obtaining an approximate $L_1$ regularizer. We therefore examine Algorithm~\ref{alg} for the piecewise constant function indicated in Figure~\ref{data1}, replacing source \eqref{eq:gravf} for the \texttt{gravity} problem with $d=0.25$. Two noise levels  and  regularization parameters $\sub{\lambda}$ found using $n=50$, $100$, $200$, $500$ and $1000$ points, with $\epsilon=10^{-15}$,  are used to obtain the solutions with $1000$ points. The presented  solutions in Figure~\ref{discontdata} are consistent with expectations from the results presented in section~\ref{results_1}. While we would not expect to resolve the discontinuities with a smoothing regularizer, the results with $10\%$ noise are sufficiently  encouraging to support  future study of the downsampling techniques  in the context of  iterative edge enhancing regularizers. 
 \begin{figure}
\begin{center}
\subfigure[Data\label{data1}] {\includegraphics[width=.18\textwidth,height=.3\textwidth]{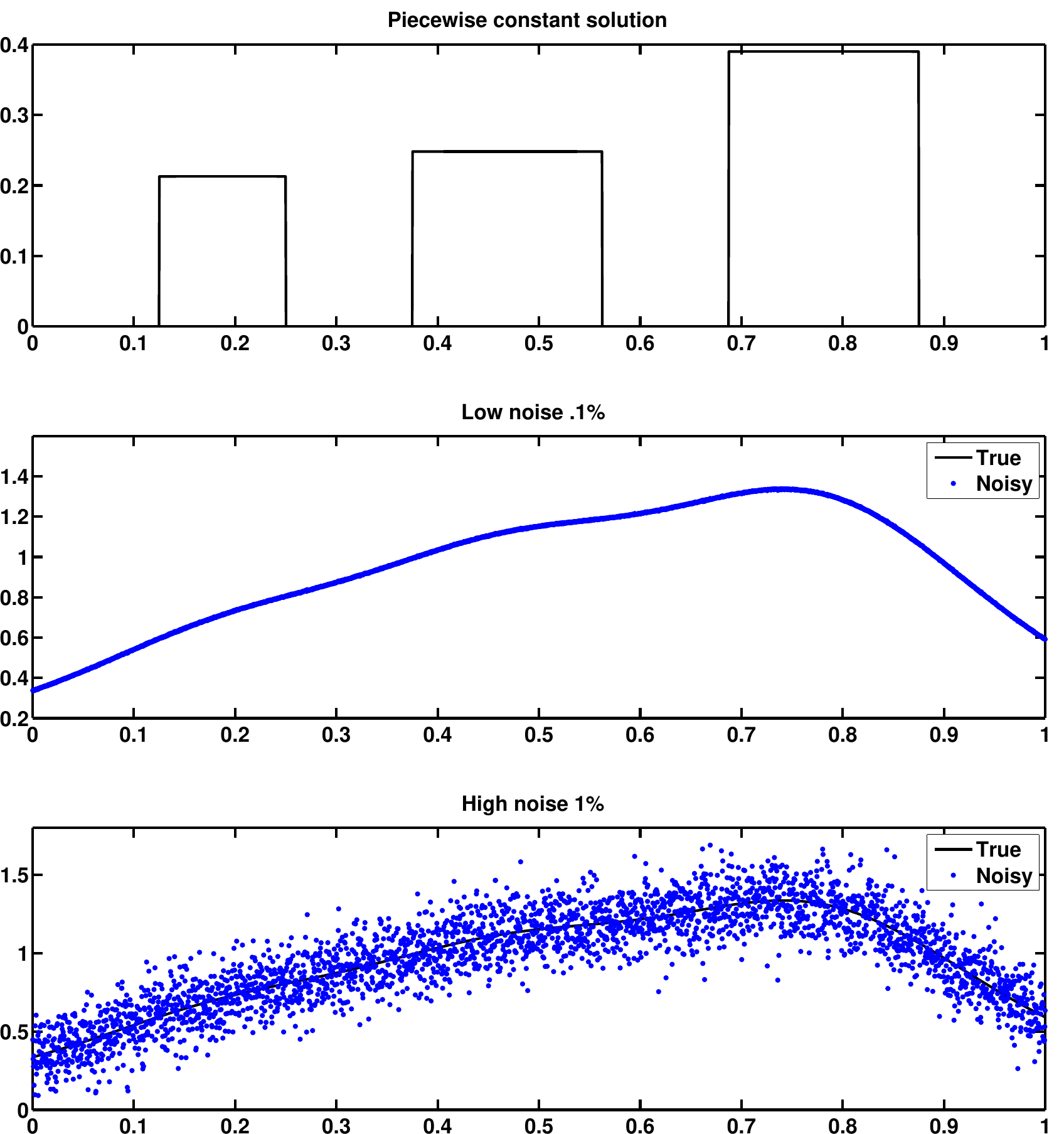}}
\subfigure[Low Noise $d=0.25$, $.1\%$ ]{\includegraphics[width=.4\textwidth, height=.3\textwidth]{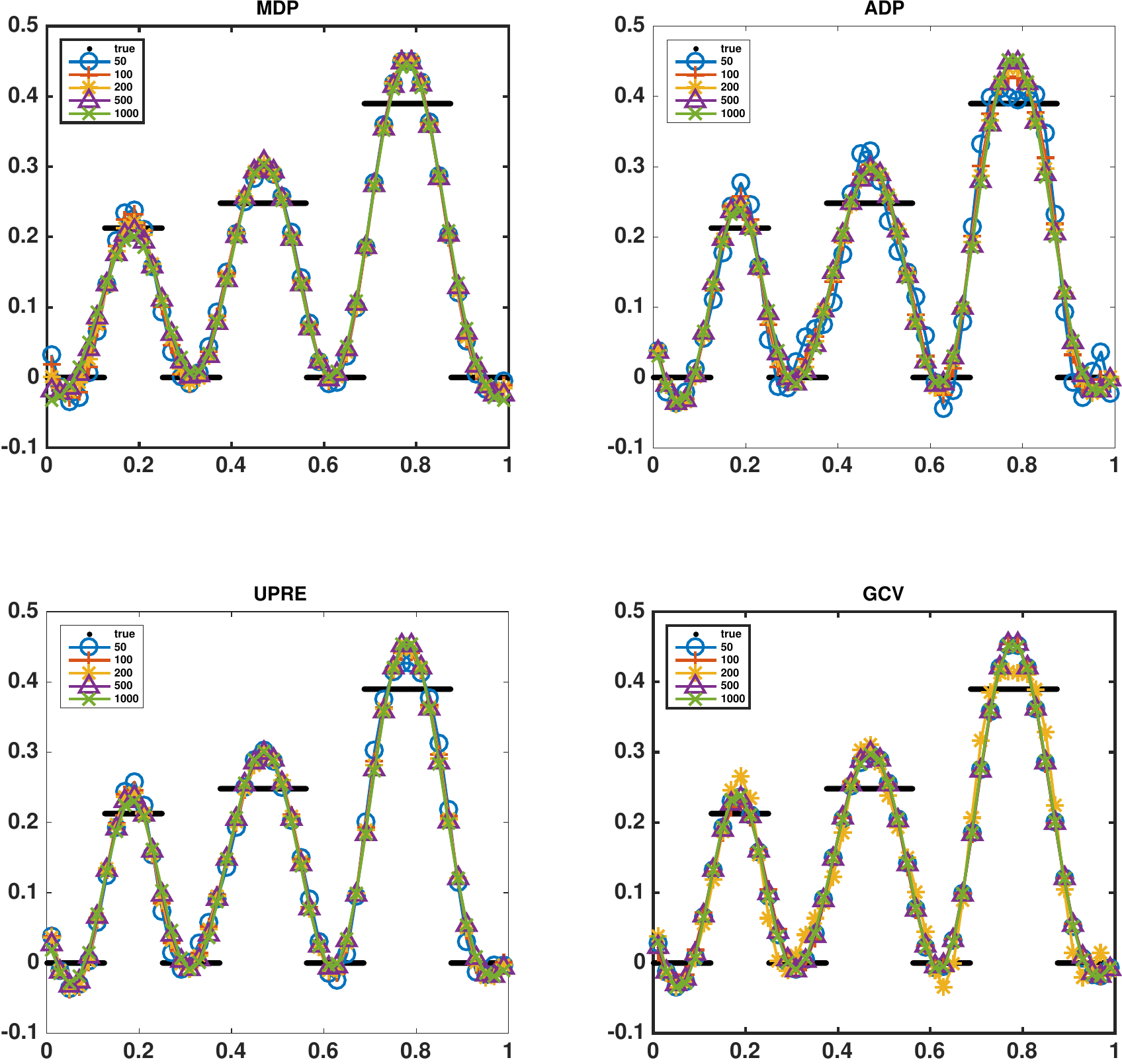}}
\subfigure[High Noise $d=0.25$, $10\%$ ]{\includegraphics[width=.4\textwidth, height=.3\textwidth]{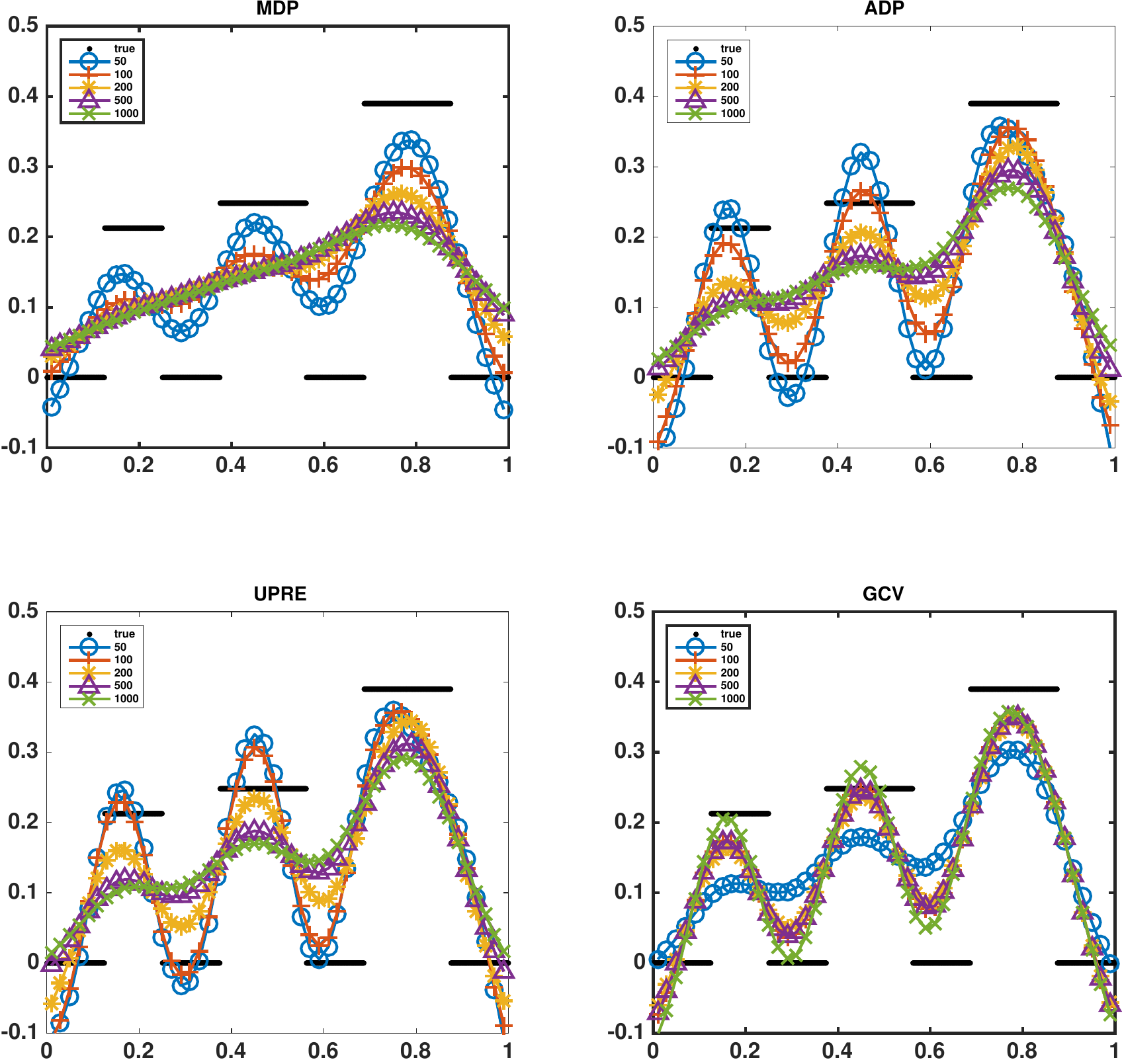}}
\end{center}
\caption{The  solutions for the noisy data illustrated in Figure~\ref{data1} with the regularization parameter found for $n=50$, $100$, $200$, $500$ and $1000$, using $\eps=10^{-15}$ for determining $p^*$, and $\zeta^2=ds^{(50)}$ and the reconstructed solution in each case for $n=1000$.  \label{discontdata}}
\end{figure}

\subsection{Slowly Decaying Spectrum}
Problem \texttt{gravity} has a quickly decaying spectrum, as shown in Figure~\ref{fig:pstar}. We now consider a  problem with a very slowly decaying spectrum, see Figure~\ref{fig:pstarderiv}, example \texttt{deriv2} from \cite{Regtools}, for which the kernel is also square integrable, see Figure~\ref{fig:deltaderiv},  $\|H(s,t)\|^2=1/90$, where $H(s,t)$ is defined on $[0,1]$ for both variables and $H(s,t)=s(t-1)$ for $s<t$ and $H(s,t)=t(s-1)$, otherwise. For this example we look at problem sizes $750$, $1000$, $1200$, $1500$, $2000$, $3000$ and $6000$. It is evident from Figure~\ref{fig:pstarderiv} that we cannot completely capture the spectrum for $N=6000$ by using smaller $n$; although as $n$ increases the spectral values closely follow spectral components for $N=6000$, for almost all terms obtained.  Picking a numerical rank is now relevant, and will exclude terms from the $N=6000$ expansion. We illustrate solutions obtained for $10\%$ and $25\%$ noise as shown in Figures~\ref{fig:pstardata}-\ref{fig:pstardata1} for exact source $f(t)=t$ for $t<0.5$ and $f(t)=(1-t)$, otherwise. The solutions obtained using  UPRE  and GCV for  different numerical ranks, $\epsilon$ decreasing from $10^{-5}$ to $10^{-8}$, corresponding to $\subl{p}$ approximately $100$, $320$, $1020$ and $3750$,  are given in Figures~\ref{deriv2soln1}-\ref{deriv2soln2} for the two noise levels. Although the spectrum decays slowly, the Picard plots, Figure~\ref{picardderiv}-\ref{picardderiv1}, show that noise enters the solution quickly for small indices, thus demonstrating that it is sufficient to use low rank, when using a single parameter estimation technique. Should one use a multi-parameter regularization one may be able to account for different windows in the spectrum, as presented in recent literature, \cite{ChEaOl:11,LSY,meadmulti}.

\begin{figure} 
\begin{centering}
\subfigure[$ \lvert \paren{{\Delta^{(n)}}}^2 \rvert $\label{fig:deltaderiv}]{\includegraphics[width= .15\textwidth]{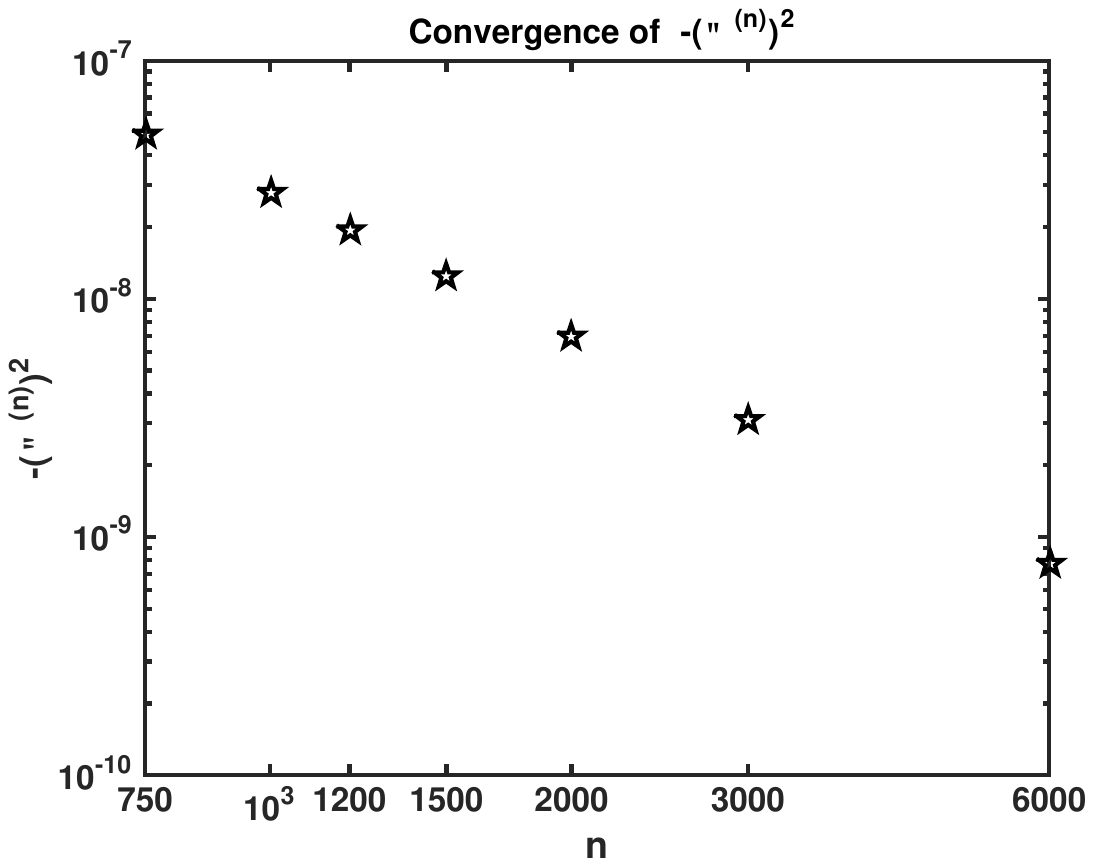}}
\subfigure[Spectrum \label{fig:pstarderiv}]{\includegraphics[width= .15\textwidth]{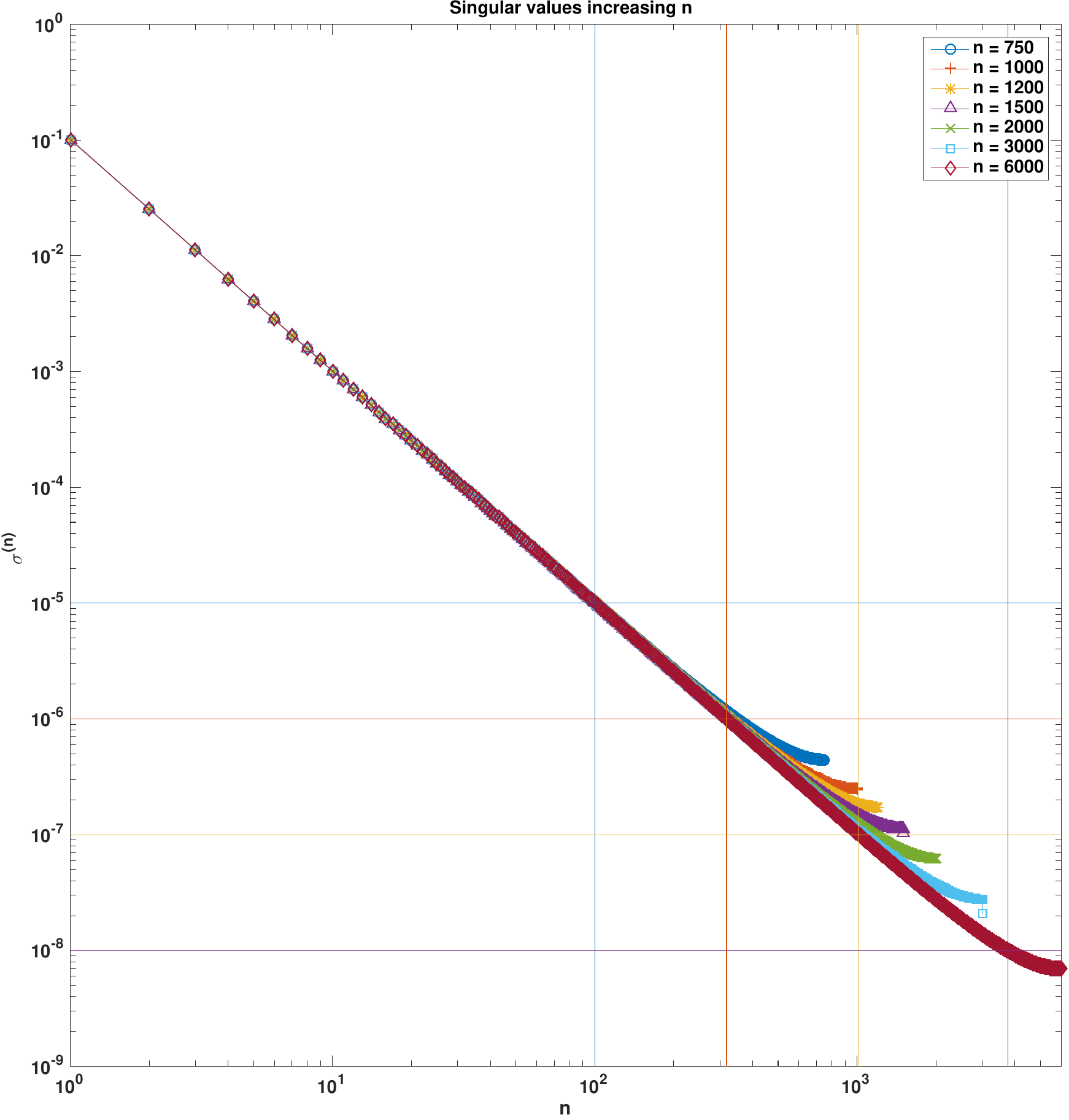}}
\subfigure[Data $10\%$ \label{fig:pstardata}]{\includegraphics[width= .15\textwidth]{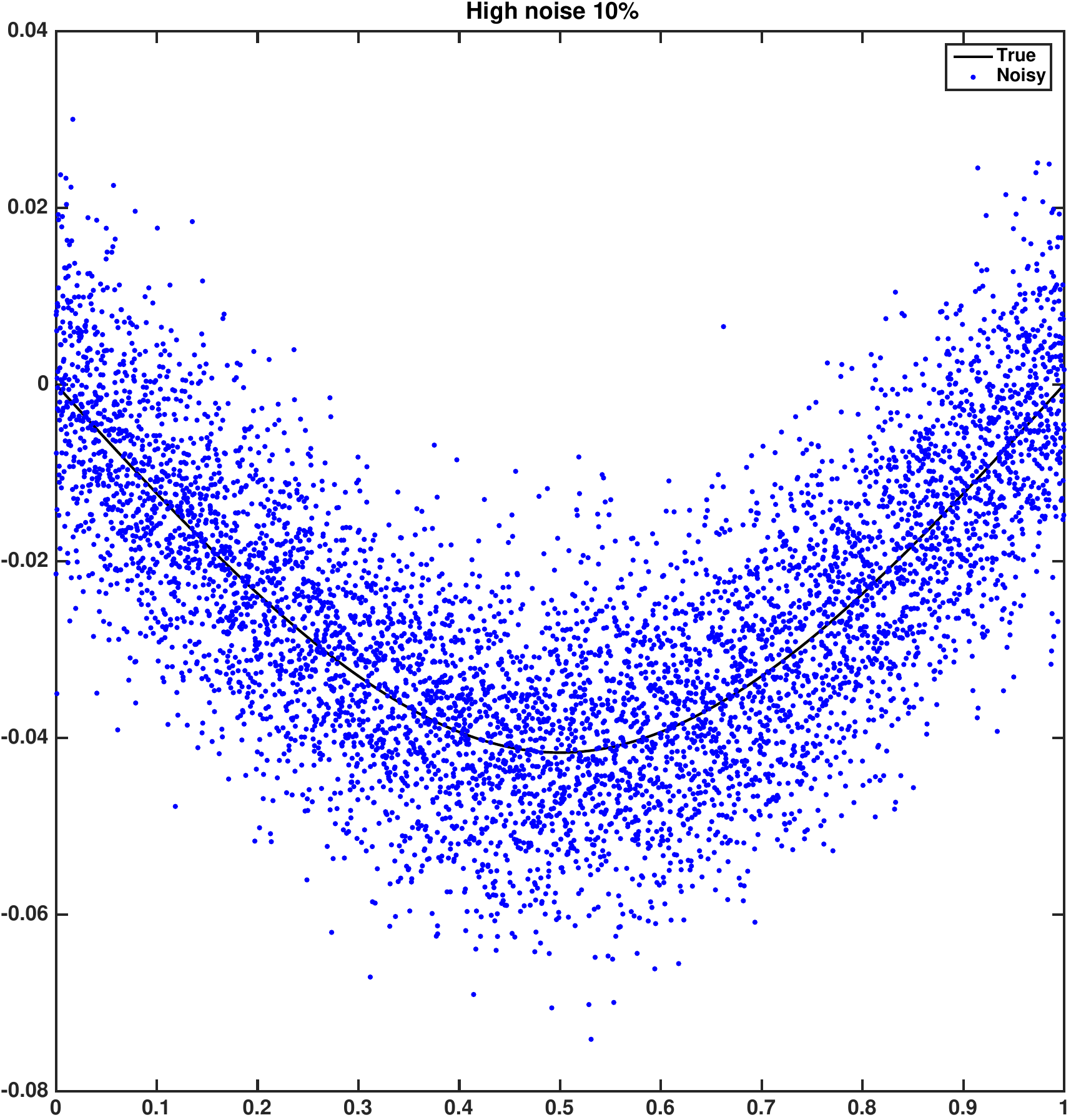}}
\subfigure[Data $25\%$\label{fig:pstardata1}]{\includegraphics[width= .15\textwidth]{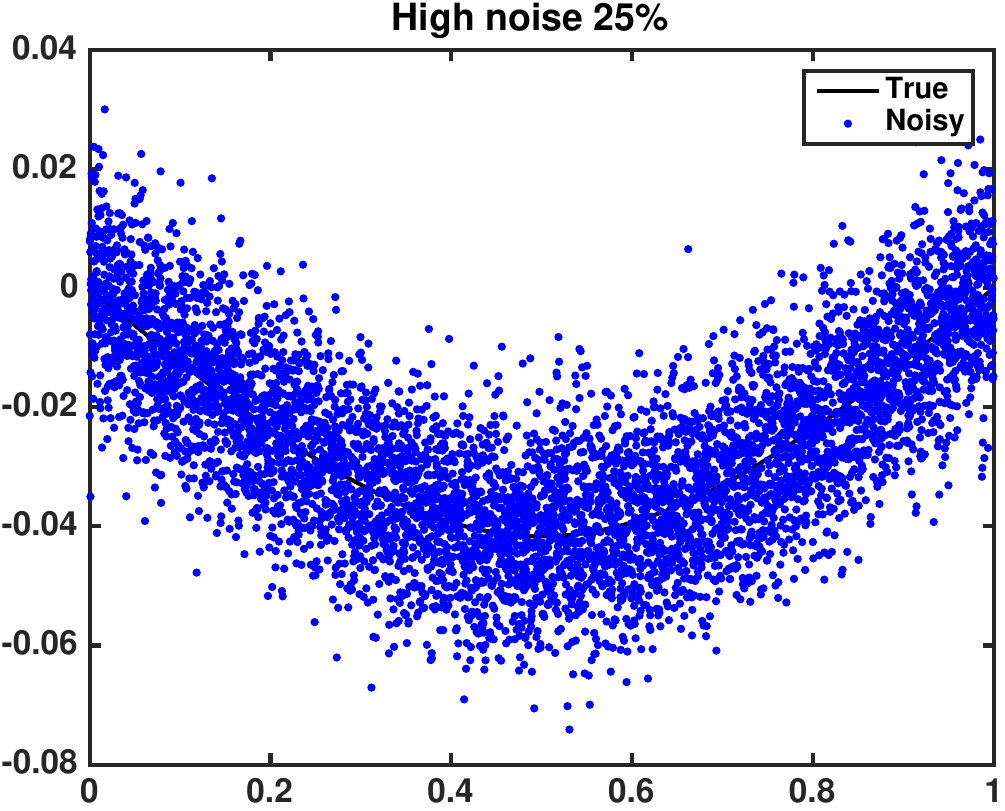}}
\subfigure[Picard $10\%$\label{picardderiv}]{\includegraphics[width= .15\textwidth]{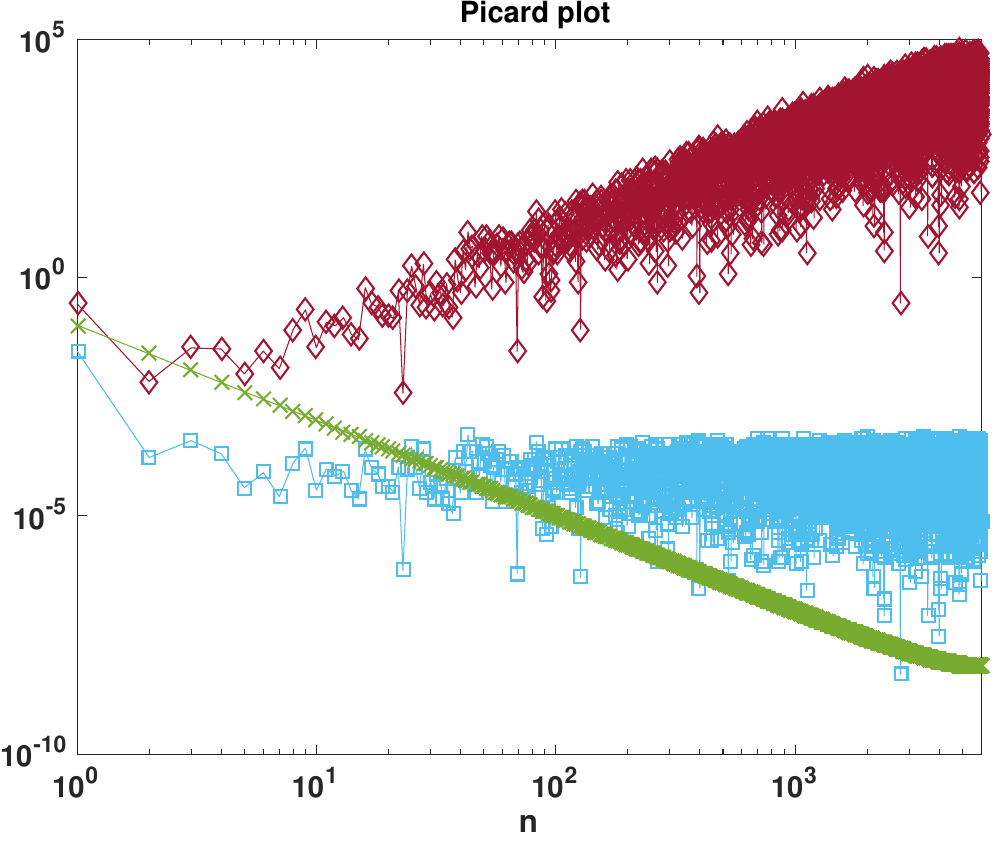}}
\subfigure[Picard $25\%$\label{picardderiv1}]{\includegraphics[width= .15\textwidth]{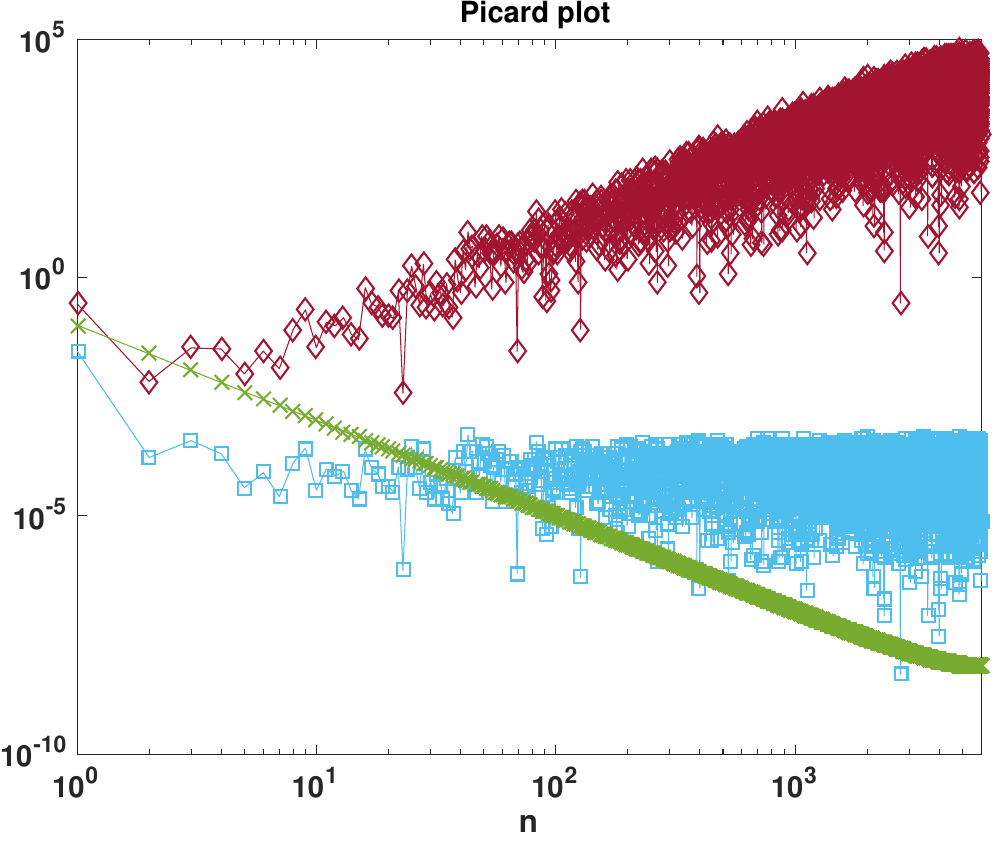}}
\caption{In \ref{fig:deltaderiv} $ \lvert \paren{{\Delta^{(n)}}}^2 \rvert $ against $n$ for problem \texttt{deriv2}. In \ref{fig:pstarderiv} the singular values plotted  against $n$ with the location of $p^*$ for $\epsilon=10^{-5}$, $10^{-6}$, $10^{-7}$ and $10^{-8}$. In \ref{fig:pstardata}-\ref{fig:pstardata1} the data with $10\%$ and $25\%$noise, and in \ref{picardderiv}-\ref{picardderiv1} the Picard plots for this data.  In these, and subsequent figures, the markers and colors are consistently determined by resolution $n$. 
}
\end{centering}
\end{figure}


\begin{figure}[H] 
\begin{centering}
\subfigure[UPRE for \texttt{deriv2} decreasing numerical rank]{\includegraphics[width= .45\textwidth]{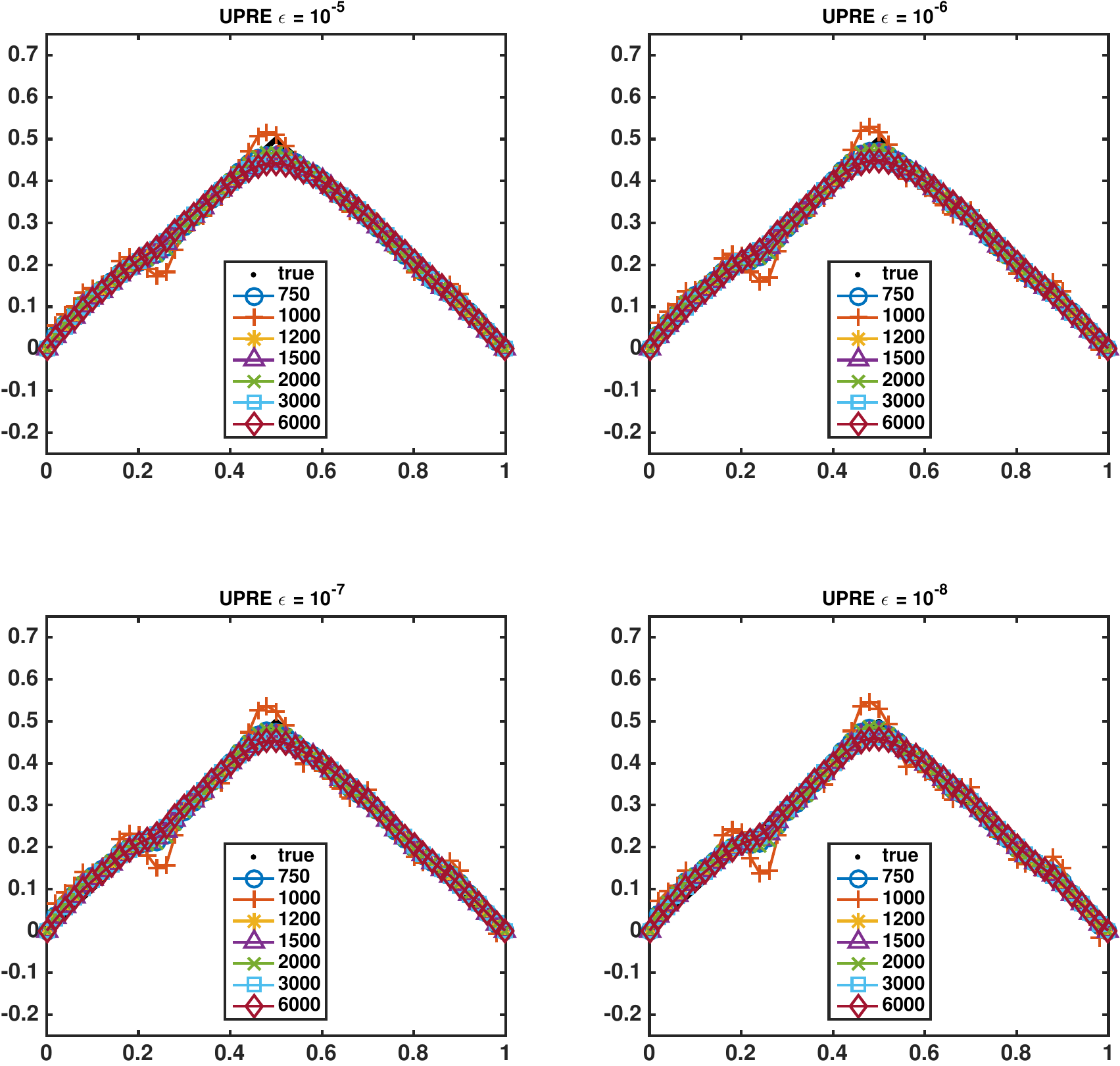}}
\subfigure[GCV for \texttt{deriv2} decreasing numerical rank]{\includegraphics[width= .45\textwidth]{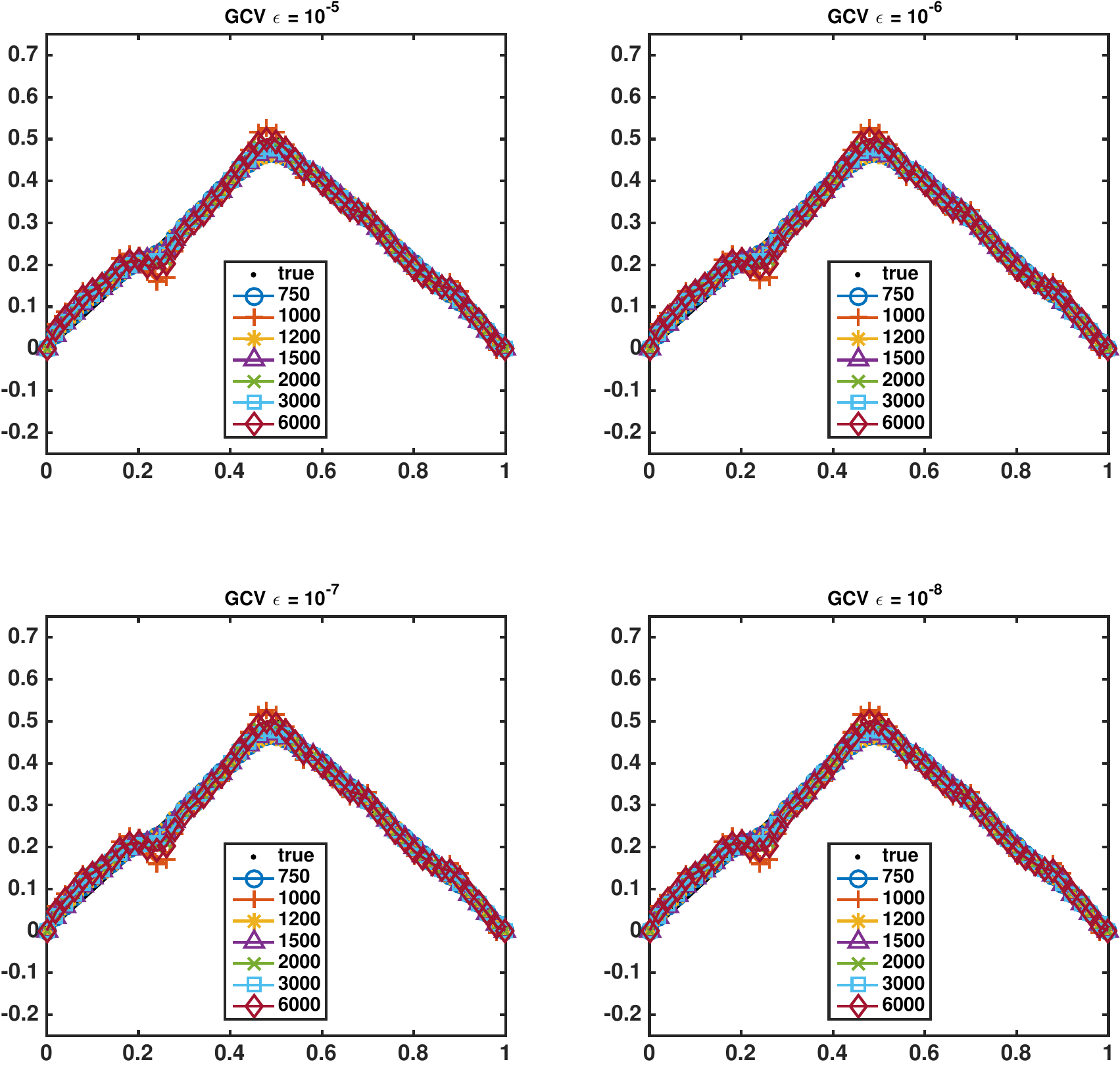}}
\caption{Solutions for problem \texttt{deriv2} with $10\%$ noise, as illustrated in Figure~\ref{fig:pstardata}, with $\lambda^{(n)}$ estimated by Algorithm~\ref{alg} and solutions reconstructed for $N=6000$, using $\zeta^2=ds^{(n)}$.  The four subplots for each method indicate solutions with decreasing numerical rank for $\epsilon$ from $10^{-5}$ to $10^{-8}$. \label{deriv2soln1}}
\end{centering}
\end{figure}
\begin{figure}[H] 
\begin{centering}
\subfigure[UPRE for \texttt{deriv2} decreasing numerical rank]{\includegraphics[width= .45\textwidth]{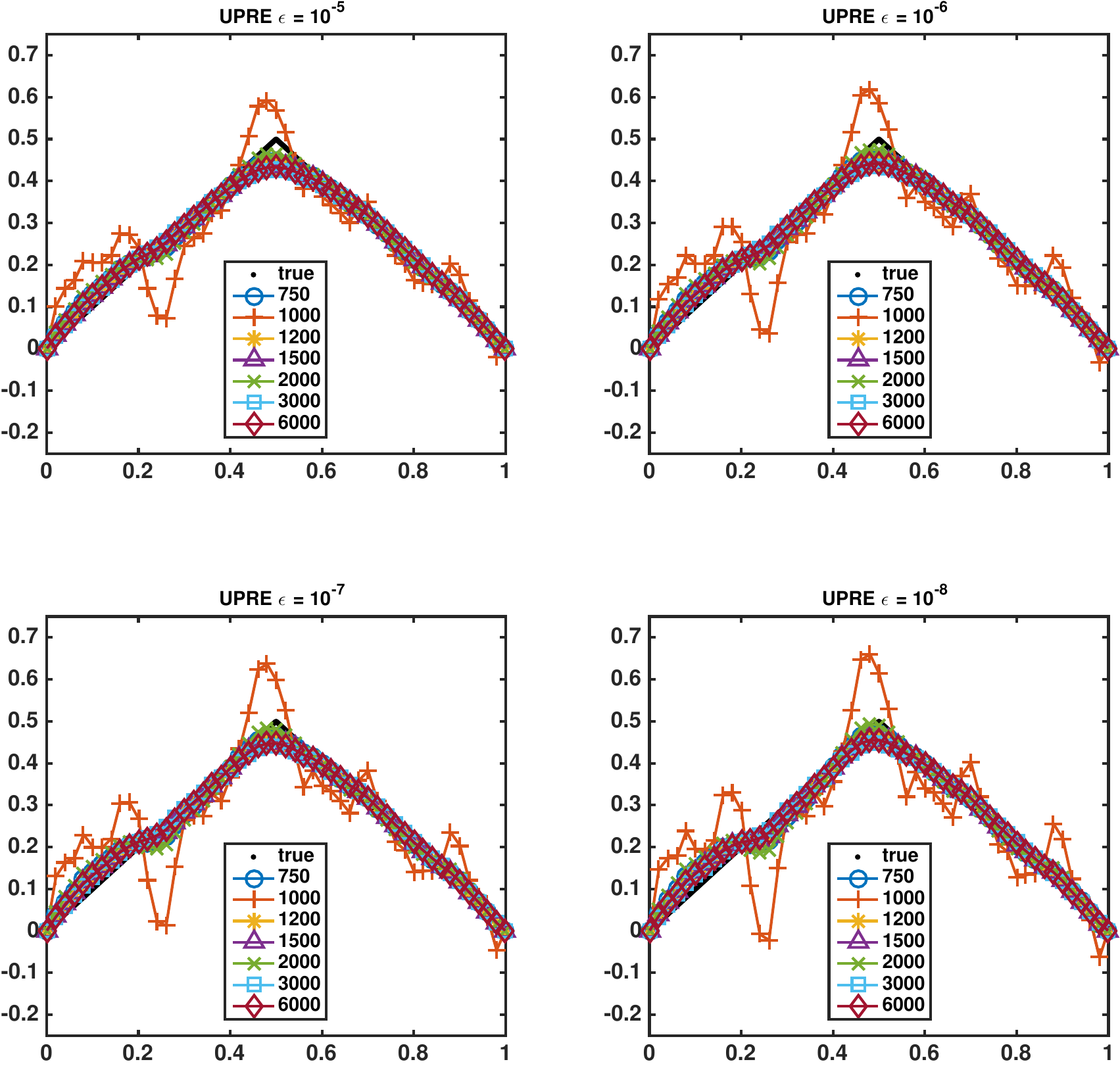}}
\subfigure[GCV for \texttt{deriv2} decreasing numerical rank]{\includegraphics[width= .45\textwidth]{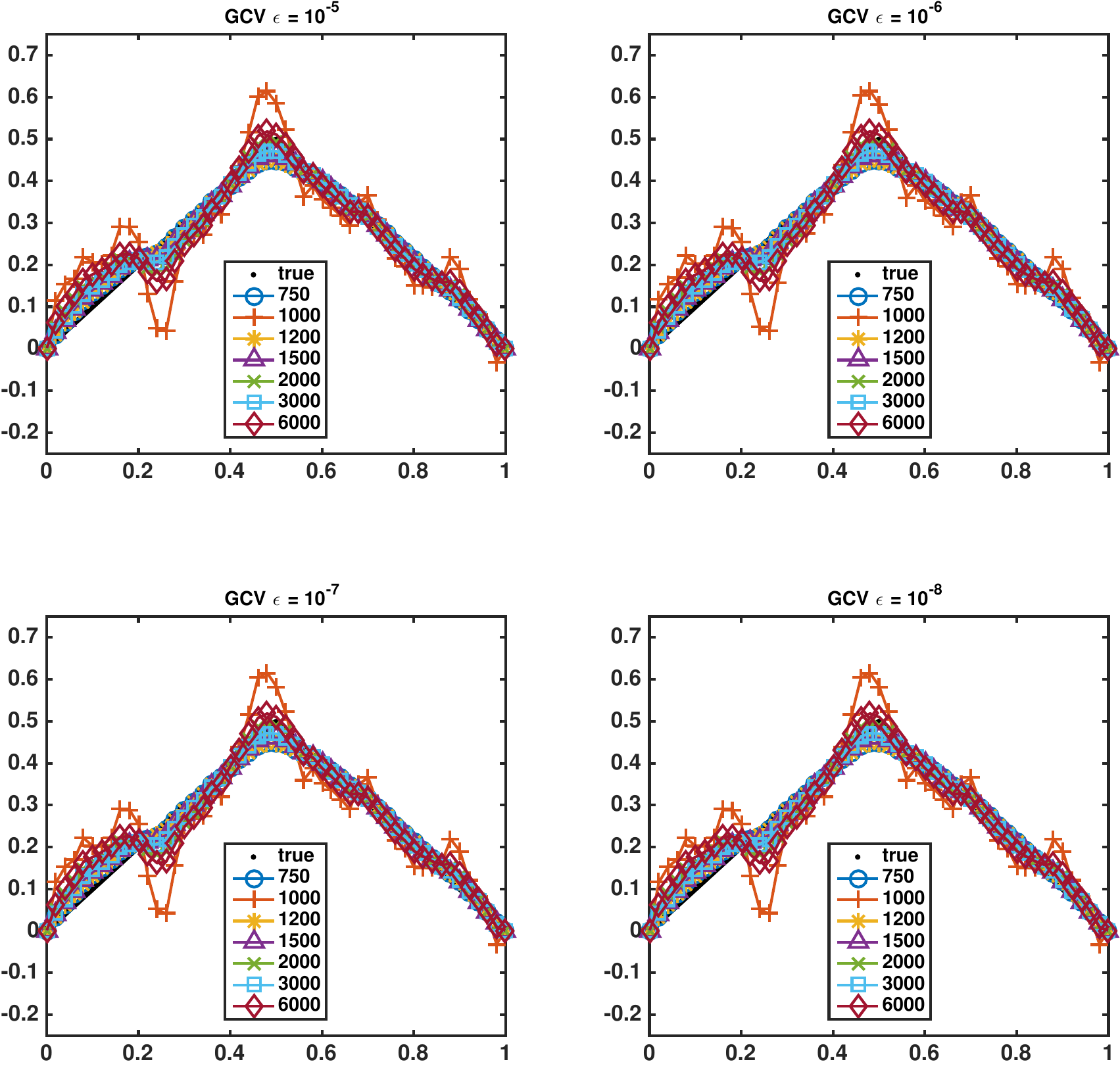}}
\caption{Solutions for problem \texttt{deriv2} with $25\%$ noise, as illustrated in Figure~\ref{fig:pstardata1}, with $\lambda^{(n)}$ estimated by Algorithm~\ref{alg} and solutions reconstructed for $N=6000$, using $\zeta^2=ds^{(n)}$.  The four subplots for each method indicate solutions with decreasing numerical rank for $\epsilon$ from $10^{-5}$ to $10^{-8}$. \label{deriv2soln2}}
\end{centering}
\end{figure}
\section{Conclusions and Future Work}\label{conclusions}
We have verified  that the theoretical relationship between the continuous SVE for   a square integrable kernel and the SVD of the  discretization of an integral equation using the Galerkin method can be exploited in the context of efficient regularization parameter selection in solving an ill-posed inverse problem. Analysis of the regularization techniques demonstrates convergence of the regularization parameter with increasing resolution for the discretization of the integral equation using the Galerkin approach. By finding the regularization parameter for a coarse representation of the system, the cost of finding the regularization parameter is negligible as compared to the solution of a fine scale problem. Moreover, exploiting numerical rank, which is  approximately preserved across resolutions for a sufficiently sampled high resolution system, mitigates the need to find the singular value decomposition for the high resolution system.  Effectively, the solution of the fine scale problem is found by projection to a coarser scale  space for the solution, on which the  dominant singular properties of the high resolution system are preserved. This provides a valid  alternative to applying a  Krylov iterative technique for the solution of the system of equations, which also uses projection to a smaller space that effectively also preserves the singular system properties from the fine scale, first presented in \cite{PaigeSau2,PaigeSau3} but by now extensively studied in amongst others \cite{ChNaOl:08,GaNoRu,HaJe,HPS,JeHa,ReSgYe,RHM:10}.

Numerical results verify both the theoretical developments and application of the technique for the solution of ill-posed integral equations, for kernels with different conditioning, and solutions which are  smooth or  piece wise constant. Although we acknowledge that Tikhonov regularization is not the method of choice for the determination of solutions which are non smooth, we note that most techniques that pose the regularization with a more relevant norm such as the $L_1$-norm \cite{SB,vogel:02,WoRo:07,Zhd}, still embed within the solution technique the need to solve using an $L_2$ norm, albeit with an operator $L$, e.g.  \eqref{genregsoln}. By judicious choice of boundary conditions, even when $L$ approximates a derivative, it is possible to approximate the derivative with an invertible operator $L$, \cite{DoRe,Lothar:2010}. Thus the application of the techniques in this paper to edge enhancing regularization is a topic for future research. 

While Vogel \cite{vogel:02} had previously provided an analysis of the convergence of regularization parameter selection techniques, including for the GCV, UPRE and MDP, he did not exploit the numerical rank to overall reduce the computational cost, in the context of the Galerkin approximation for the integral equation.  Moreover, a discussion of the extension of  these techniques for higher dimensions was not addressed. For separable and square integrable kernels, the tensor product SVD may be   used to further reduce the computational cost, with different rank estimations in each dimension of the kernel. These ideas are also relevant in the context of spatially invariant kernels which admit convolutional representations for the integral equation, and dependent on the boundary conditions, can be solved using Fourier or cosine transforms \cite{vogel:02,hansenbook,HNO}. Again these are topics of future research, with application for practically relevant large scale problems. Further, in the context of practically relevant large scale problems, an algorithmic approach to assessing sufficient convergence of the solution for increasing resolution should be developed.

\ack  
Rosemary Renaut acknowledges the support of AFOSR grant 025717: ``Development and Analysis of Non-Classical Numerical Approximation Methods", and 
NSF grant  DMS 1216559:   ``Novel Numerical Approximation Techniques for Non-Standard Sampling Regimes".


\begin{thebibliography}{00}

\bibitem{ABT} Aster R C, Borchers B  and Thurber C H  2013  \textit{Parameter Estimation and Inverse Problems}  second edition Elsevier Inc. Amsterdam 
\bibitem{baker}
Baker C T H 1977  \textit{The Numerical Treatment of Integral Equations}  Clarendon Press 
\bibitem{ChNaOl:08} Chung J M Nagy J and O'Leary D P 2008 {A weighted GCV method for Lanczos hybrid regularization} \textit{ETNA},  {\bf 28}, 149-167

\bibitem{DoRe} Donatelli M and Reichel L 2014 {Square smoothing regularization matrices with accurate boundary conditions} \textit{J Computational and Applied Mathematics} {\bf 272} 334-349  

\bibitem{ChEaOl:11} Chung J   Easley G  and  O'Leary D P 2011 {Windowed Spectral Regularization of Inverse Problems}  \textit{SIAM Journal on Scientific Computing} {\bf 6} 3175-3200


\bibitem{GaNoRu} Gazzola S, Novati P and Russo M R 2014 {Embedded techniques for choosing the parameter in {T}ikhonov regularization} \textit{Numerical Linear Algebra and Applications} {\bf 21} 6 796-812
\bibitem{SB} 
 Goldstein  T and Osher  S 2009 \textit{The split Bregman method for L1-regularized problems} 
  {SIAM J. Img. Sci.}   {\bf 2} 2 323-343 
  
\bibitem{GoLo:96} Golub G H and van Loan C 1996 \textit{Matrix Computations}  3rd ed.  Johns Hopkins Press Baltimore 

\bibitem{golub1979generalized} Golub G H Heath M and Wahba G  1979
\textit{Generalized cross-validation as a method for choosing a good ridge parameter}  
 {Technometrics} 
 \textbf{21} 2 
{215--223}

\bibitem{hansensvepaper}
Hansen P C  1988 \textit{Computation of the singular value expansion}  Computing \textbf{40}  185-199 
\bibitem{Hansen} Hansen P C 1998  \textit{Rank-Deficient and Discrete Ill-Posed Problems: Numerical Aspects of Linear Inversion}  SIAM Monographs on Mathematical Modeling and Computation  \textbf{4}  Philadelphia 
\bibitem{HansenLC} Hansen P C 2001 \textit{The L-curve and its use in the numerical treatment of inverse problems}  Invited chapter in Computational Inverse Problems in Electrocardiology  P Johnston  ed. WIT Press  Southampton  119-142   
\bibitem{Regtools} Hansen P C 2007  \textit{Regularization Tools:A Matlab package for analysis and solution of discrete ill-posed problems. Version 4.0 for Matlab 7.3}  Numerical Algorithms   {\bf 46}   189-194  and \url{http://www2.imm.dtu.dk/~pcha/Regutools/} 

\bibitem{hansenbook}
Hansen P C 2013 \textit{Discrete Inverse Problems: Insights and Algorithms}  SIAM Series on Fundamentals of Algorithms \textbf{7}  Philadelphia  PA: SIAM 
\bibitem{HaJe} Hansen P C and Jensen T K 2008 {Noise propagation in regularizing iterations for image deblurring} \textit{ETNA} {\bf 31} 204-220  

\bibitem{HNO}  Hansen P C Nagy J and O'Leary D 2006  \textit{Deblurring Images Matrices Spectra and Filtering}  Philadelphia  PA: SIAM
\bibitem{HPS} Hn\u etynkov\'a  I    Ple\u singer  M  and Strako\u s, Z 2009 The regularizing effect of the Golub-Kahan iterative bidiagonalization and revealing the noise level  in the data  \textit{BIT Numerical Mathematics} {\bf 49} 4 669-696 

\bibitem{JeHa} Jensen T K and Hansen P C 2007 Iterative regularization with minimum-residual methods \textit{BIT Numerical Mathematics} {\bf 47} 103-120 
\bibitem{Lothar:2010} Hochstenbach M E and Reichel L 2010 An iterative method for Tikhonov regularization with general linear regularization operator \textit{J. Integral Equations Appl.} {\bf 22} 463-480

\bibitem{LSY} Lu Y  Shen L and  Xu Y 2007 \textit{Multi-parameter regularization methods for high-resolution image reconstruction with displacement errors}  IEEE Transactions on Circuits and Systems I  {\bf 54}  8  1788-1799
 
\bibitem{meadmulti}
 Mead J L  2013  {\textit{Discontinuous parameter estimates with least squares estimators}} Applied Mathematics and Computation  {\bf 219} 5210-5223
\bibitem{mere:09}
Mead J L and Renaut R A 2009 \textit{A Newton root-finding algorithm for estimating the regularization parameter for solving ill-conditioned least squares problems}  Inverse Problems \textbf{25} 025002 


\bibitem{morozov} Morozov V A 1966 \textit{On the solution of functional equations by the method of regularization}   {Sov. Math. Dokl.} {\bf 7} 414-417 

\bibitem{PaigeSau2}Paige C C and Saunders M A 1982 LSQR: An algorithm for sparse linear equations and sparse least squares \textit{ACM Trans. Math. Software}  {\bf 8} 43-71
\bibitem{PaigeSau3}Paige C C and Saunders M A 1982 ALGORITHM 583 LSQR: Sparse linear equations and least squares problems \textit{ACM Trans. Math. Software}  {\bf 8}  195-209
\bibitem{ReSgYe}  Reichel L  Sgallari F and  Ye Q 2012 Tikhonov regularization based on generalized Krylov subspace methods
\textit{Appl. Numer. Math.}, {\bf  62} 1215-1228 
\bibitem{RHM:10}
Renaut R A  Hn\v{e}tynkov\'{a} I and Mead  J L 2010 \textit{Regularization parameter for large-scale Tikhonov regularization using a priori information}  Computational Statistics and Data Analysis  \textbf{54}  3430-3445 

\bibitem{smithies}
Smithies  F  1958  \textit{Integral Equations}  Cambridge Tract No. 49  Cambridge University Press 


\bibitem{vogel:02}
Vogel C R  2002 \textit{Computational Methods for Inverse Problems}  SIAM Frontiers in Applied Mathematics  Philadelphia  PA: SIAM  
\bibitem{WoRo:07} Wohlberg B and Rodriguez P 2007 \textit{An ieratively reweighted norm algorithm for minimization of total variation functionals}   IEEE Signal Processing Letters  {\bf 14} 948--951 
\bibitem{Zhd} Zhdanov M S 2002 \textit{Geophysical Inverse Theory and Regularization Problems} Amsterdam The Netherlands: Elsevier Inc. 
 
\end{thebibliography}
\end{document}